\documentclass[11pt]{amsart}
\usepackage{amssymb, latexsym, graphicx, mathrsfs, enumerate, amsmath, amsthm, textcomp, url, tikz-cd,bbm}
\usepackage[T1]{fontenc}
\usepackage{mathdots}
\usepackage{comment}
\usepackage[toc,page]{appendix}
\usepackage{cases}
\usepackage{enumitem}
\usepackage{hyperref}
\hypersetup{
    colorlinks=true,
    linkcolor=blue,
    filecolor=magenta,      
    urlcolor=cyan,
}

\usepackage{color}
\usepackage{hyperref}
\hypersetup{
    colorlinks=true,
    linkcolor=blue,
    filecolor=magenta,      
    urlcolor=cyan,
}

\input xy
\xyoption{all}

\allowdisplaybreaks
\allowdisplaybreaks

\numberwithin{equation}{section}
\newtheorem{Thm}{Theorem}[section]
\newtheorem{Prop}[Thm]{Proposition}
\newtheorem{remark}[Thm]{Remark}
\newtheorem{Cor}[Thm]{Corollary}
\newtheorem{Lem}[Thm]{Lemma}

\theoremstyle{definition}
\newtheorem{Def}[Thm]{Definition}

\newtheorem{Rmk}[Thm]{Remark}

\newcommand{\mfo}{\mathfrak{o}}

\begin{document}

\title[On the Siegel series in terms of lattice counting]
{On the Siegel series in terms of lattice counting}


\author[Sungmun Cho]{Sungmun Cho}
\author[Taeyeoup Kang]{Taeyeoup Kang}
\thanks{The authors are supported by  Samsung Science and Technology Foundation under Project Number SSTF-BA2001-04.}

\address{Sungmun Cho \\  Department of Mathematics, POSTECH, 77, Cheongam-ro, Nam-gu, Pohang-si, Gyeongsangbuk-do, 37673, KOREA}

\email{sungmuncho12@gmail.com}

\address{Taeyeoup Kang \\  Department of Mathematics, POSTECH, 77, Cheongam-ro, Nam-gu, Pohang-si, Gyeongsangbuk-do, 37673, KOREA}

\email{taeyeoupkang@gmail.com}

\maketitle

\begin{abstract}
In this paper we describe each coefficient of the Siegel series associated to a quadratic $\mfo$-lattice $L$   in terms of lattice counting problems,  
where $\mfo$ is the ring of integers of a non-Archimedean local field of characteristic $0$. 
Under the restriction that $p$ is odd and that the dimension of  the radical of the quadratic space $L\otimes\kappa$ on the residue field $\kappa$ is  at most $2$, we provide explicit values of coefficients  and reprove the functional equation of the Siegel series. 
\end{abstract}

 \tableofcontents

\section{Introduction}\label{in}

The Siegel series is a central object  in the study of theories of automorphic forms and $L$-functions.
It is related to the Fourier coefficients of the Siegel Eisenstein series and to the construction of automorphic forms by Ikeda lift.
Katsurada in \cite{Kat} firstly obtained the explicit formula of the Siegel series. 
Ikeda and and Katsurada in \cite{IK2} extended this result to any finite field extension of $\mathbb{Q}_p$.
On the other hand, the functional equation of the Siegel series over a finite field extension of $\mathbb{Q}_p$ is obtained by Ikeda in \cite{I1}.
The functional equation of the Siegel series  is a key ingredient to main results of  \cite{Kat} and \cite{IK2}.
For more history  of the Siegel series, we refer to the first few paragraphs of the introduction of \cite{IK2}.

Meanwhile, one of  authors of our manuscript and Yamauchi proposed a new method toward the Siegel series in \cite{CY}:
they interpreted the Siegel series as the volume of a certain $p$-adic manifold and  then analyzed this volume using smoothening method, which is a desingularization of a certain scheme defined over a discrete valuation ring.
Through this method, they found a new description of the Siegel series in terms of weighted lattice counting problems, where each weight is explicitly described.
This interpretation plays an important role in the proof of the Kulda-Rapoport conjecture by Li and Zhang in \cite{LZ1} and \cite{LZ2}.

A main goal of this paper is to develop observation of \cite{CY} to further analyze the Siegel series.
Our work is summarized as follows:
\begin{enumerate}
    \item 
We describe each coefficient of the Siegel series in terms of lattice counting problems without any restriction on $p$ in Theorem \ref{thm:func_eq_coeff}.
Here  the Siegel series is known to be expressed  as a polynomial in $\mathbb{Z}[X]$. 

    \item We  provide explicit values of coefficients of the Siegel series by solving lattice counting problems (cf. Proposition \ref{prop:Sab}), under a certain   restriction in Corollaries \ref{cor:coeffofct}-\ref{cor:coeff_n0=n-2} and in Theorem \ref{thm45}.
Here ``a certain restriction'' will be explained below more precisely.
    
    \item We  reprove the functional equation of the Siegel series under the same restriction in Corollary \ref{cor:functionaleq}. 
Since the Siegel series is expressed by a polynomial, the functional equation means comparison among coefficients. 

\end{enumerate}

We will explain our main results more precisely.
Let $L$ be a quadratic $\mfo$-lattice of rank $n$, where $\mfo$ is the ring of integers of a non-Archimedean local field of characteristic $0$. 
A main contribution of \cite{CY} is to express the Siegel series associated to $L$, denoted by $F_L(X) (\in \mathbb{Z}[X])$,  as a sum of  polynomials 
whose coefficients involve weighted sum of the number of certain lattices, denoted by  $\# \mathcal{S}_{(L,a^\pm,b)}$ (cf. Proposition \ref{prop2.4} and Definition \ref{def:siegelseries}).
Here 
\begin{itemize}
    \item The index set of the sum of polynomials is determined by the Gross-Keating invariant of $L$, where the definition of the Gross-Keating invariant is given in Definition \ref{gkdef}. 

    \item $\# \mathcal{S}_{(L,a^\pm,b)}$ is the number of integral quadratic $\mfo$-lattices $L'$ containing $L$  with $a=$ the dimension of the orthogonal complement of the radical of the quadratic space $L'\otimes\kappa$ and with $b=[L':L]$, where $\kappa$ is  the residue field of $\mfo$ and $[L':L]$ is the length of the torsion $\mfo$-module $L'/L$.
This is explained in Remark \ref{rmk:gk}.
\end{itemize}

In Theorem \ref{thm:func_eq_coeff} we sum up these polynomials so as to describe each coefficient of the Siegel series using $\# \mathcal{S}_{(L,a^\pm,b)}$.
In other words, if we let $F_L(f^{-(n+1)/2}X)=\sum_{t\geq 0}c_tX^t$, where $f=\#\kappa$, then we describe $c_t$ as a weighted sum of $\# \mathcal{S}_{(L,a^\pm,b)}$'s. 

In Corollary \ref{cor:coeffofct}, we describe $c_0, c_1, c_2, c_3$ more precisely. 
Note that $c_0=1$ is a well-known fact. An interesting result of this corollary is that $c_1$ has also a simple description, by avoiding the use of $\# \mathcal{S}_{(L,a^\pm,b)}$'s. 

In Corollaries \ref{cor:coeff_n0=n-1}-\ref{cor:coeff_n0=n-2}, we describe $c_i$'s more precisely involving $\# \mathcal{S}_{(L,a^\pm,b)}$'s, when the dimension of the orthogonal complement of the radical of the quadratic space $L\otimes\kappa$ is either $n-1$ or $n-2$.

It seems that the calculation of $\# \mathcal{S}_{(L,a^\pm,b)}$ is technical. Under the restriction that $p$ is odd and that the dimension of the radical of the quadratic space $L\otimes\kappa$ is  at most $2$, we provide an explicit formula for $\# \mathcal{S}_{(L,a^\pm,b)}$ in Proposition \ref{prop:Sab}. Using this,  we describe $c_i$'s explicitly in Theorem \ref{thm45} under the same restriction. 
This directly yields the functional equation of the Siegel series in Corollary \ref{cor:functionaleq}. 

\begin{Rmk}
    \cite[Theorem 4.3]{Kat} and \cite[Theorem 5.5]{C4} provide an explicit formula for the Siegel series over $\mathbb{Q}_p$ (more precisely, over any finite field extension of $\mathbb{Q}_p$ with $p$ odd or over  any finite unramified field extension of $\mathbb{Q}_2$). 
Two crucial ingredients used to derive these formulas are the functional equation  and recursion formulas of the Siegel series, whereas our description  of the Siegel series or coefficients ((1) and (2) in the above) does not use them at all. 

Thus our argument is independent of traditional methods which can be found in the literature (of course except for \cite{CY}).
Furthermore this suggests a new interpretation on main results of \cite{I1}, \cite{Kat}, \cite{C4}, and \cite{IK2} in terms of lattice counting problems.

It will be interesting to investigate lattice counting problems by using formulas given in \cite[Theorem 4.3]{Kat} and \cite[Theorem 5.5]{C4} or vice versa. 
For example, it seems to us that the functional equation of the Siegel series is closely related to inclusion-exclusion principle among quadratic lattices. 
\end{Rmk}

We organize this paper as follows.
After fixing notations  in Section \ref{sec:not} and explaining the description of the Siegel series of \cite{CY} in Section \ref{sec:siegel},
we describe coefficients of the Siegel series in terms of lattice counting, $\# \mathcal{S}_{(L,a^\pm,b)}$,  in Section \ref{sec:3}.
In Section \ref{sec:4}, by computing $\# \mathcal{S}_{(L,a^\pm,b)}$ under the restriction that $p$ is odd and that the dimension of the radical of the quadratic space $L\otimes\kappa$ is at most $2$, we compute coefficients of the Siegel series explicitly and derive the functional equation.

\medskip

\textbf{Acknowledgments.} 
We would like to thank  Yuchan Lee  for many fruitful discussions, especially ideas on Section \ref{sec:4}.



\section{Notations}\label{sec:not}
This section is taken from \cite[Section 2]{CY}.
\begin{itemize}
\item Let $F$ be a finite extension of $\mathbb{Q}_p$ with $\mfo$  its ring of integers and $\kappa$  its residue field.
Let $\pi$ be a uniformizer in $\mfo$.
Let $f$ be the cardinality of the finite field $\kappa$.
\item For an element $x\in F$, the exponential order of $x$ with respect to the maximal ideal in $\mathfrak{o}$ is written by $\mathrm{ord}(x)$.
The value of $x$ is $|x|:=f^{-\mathrm{ord}(x)}$.

\item Let $e=\mathrm{ord}(2)$.
Thus if $p$ is odd, then $e=0$.

\item We consider an $\mfo$-lattice $L$ of rank $n$ with a quadratic form $q_L : L \rightarrow \mfo.$
 We denote by a pair $(L, q_L)$ a quadratic lattice of rank $n$.
We assume that $V=L\otimes_{\mfo}F$ is nondegenerate with respect to $q_L\otimes_{\mfo}1$.

\item For a given quadratic lattice $(L, q_L)$ over $\mfo$,
the quadratic form $\bar{q}_L$ on $L\otimes_{\mfo}\kappa$ is defined to be $q_L~mod~\pi$ and $L\otimes_{\mfo}\kappa$ is denoted by $\overline{L}$.

\item The fractional ideal generated by $q_L(X)$ as $X$ runs through $L$ will be called the \textit{norm} of $L$ and written $N(L)$.

\item Let $H=\begin{pmatrix} 0& 1/2 \\ 1/2 & 0 \end{pmatrix}$ and let $H_k$ be the orthogonal sum of the $k$-copies of $H$.
Then $H_k$ defines a quadratic lattice of rank $2k$.
We denote such a lattice by  $(H_k, q_k)$, if this does not cause confusion or ambiguity.
Similarly we define the quadratic space $(W_k,q_k\otimes 1)$, where $W_k=H_k\otimes_{\mathfrak{o}}F$.
If there is no ambiguity, then we simply use $q_k$, rather than $q_k\otimes 1$,  to stand for the quadratic form defined on $W_k$. 

\item For a symmetric matrix $B$ of size $n\times n$, we say that 
 $B$ is  half-integral over $\mathfrak{o}$ if 
each non-diagonal entry  multiplied by $2$ and each diagonal entry of $B$ are in $\mathfrak{o}$.
We sometimes say that $B$ is half-integral, by omitting `over $\mathfrak{o}$', if it does not cause ambiguity.

\item We say that $B$ is non-degenerate if the determinant of $B$ is non-zero.

\item For $U\in \mathrm{GL}_n(\mathfrak{o})$ and $B$,  where $B$ is a  half-integral symmetric matrix over $\mathfrak{o}$  of size $n\times n$,
we set  $B[U]={}^tUBU$.
Here, ${}^tU$ is the matrix transpose of $U$.

\item For a rational number $a\in \mathbb{Q}$, 
$\lfloor a\rfloor$ is the largest integer which is less than or equal to $a$ and $\lceil a \rceil$ is the smallest integer which is greater than or equal to $a$.

\item For an integer $m$, we define
    $
    \delta(m) :=
    \begin{cases}
      1 &\textit{if } m \textit{ is even}; \\
      0 &\textit{otherwise}.
    \end{cases}
    $
\item For $x,y \in F^\times$, we denote by $\left(x, y \right)$ the Hilbert symbol. 
If $p$ is odd, then \cite[Example 63:12]{MT00} yields that
$(x,\pi) = \begin{cases}
    1 &\textit{if } x \in (F^\times)^2; \\
    -1 &\textit{otherwise}.
\end{cases}$

\end{itemize}

\section{Description of the Siegel series in \cite{CY}}\label{sec:siegel}
In this section, we introduce the definition of the Siegel series and explain its description which was given in \cite{CY}.

\begin{Def}\cite[Definitions 0.1 and 0.2]{IK1}\label{gkdef}
\begin{enumerate}
\item Let $B=\begin{pmatrix}b_{ij}\end{pmatrix}$ be a non-degenerate half-integral symmetric matrix over $\mathfrak{o}$ of size $n \times n$.
Let $S(B)$ be the set of all non-decreasing sequences $(a_1, \cdots, a_n)\in \mathbb{Z}^n_{\geq 0}$ such that
\[
\begin{array}{l l}
   \mathrm{ord}(b_{ii}) \geq a_i   & \quad    \text{($1 \leq i \leq n$)};\\
   \mathrm{ord}(2b_{ij}) \geq (a_i+a_j)/2   & \quad    \text{($1 \leq i \leq j \leq n$)}.
    \end{array}\]
   Put
   \[\mathbf{S}(\{B\})=\bigcup_{U\in \mathrm{GL}_n(\mathfrak{o})}S(B[U]).\]
The Gross-Keating invariant $\mathrm{GK}(B)$ of $B$ is the greatest element of $\mathbf{S}(\{B\})$ with respect to the lexicographic order
$\succeq$ on $\mathbb{Z}^n_{\geq 0}$.
Here, the lexicographic order $\succeq$ on $\mathbb{Z}^n_{\geq 0}$  is the following (cf. the paragraph following Definition 0.1 of \cite{IK1}).
Choose two elements $(a_1, \cdots, a_n)$ and $(b_1, \cdots, b_n)$ in $\mathbb{Z}^n_{\geq 0}$.
Let $i$ be the first integer over which $a_i$ differs from $b_i$ (so that $a_j=b_j$ for any $j<i$).
If $a_i>b_i$, then we say that $(a_1, \cdots, a_n) \succ (b_1, \cdots, b_n)$.
Otherwise, we say that $(a_1, \cdots, a_n) \prec (b_1, \cdots, b_n)$.


\item Choose an half-integral symmetric matrix $B$ which represents a quadratic lattice  $(L, q_L)$.
Then we say that $\mathrm{GK}(L)=\mathrm{GK}(B)$.
This is independent of the choice of the matrix $B$.

\item Let $\mathrm{GK}(L)=(a_1, \cdots, a_n)$.
Then $|\mathrm{GK}(L)|$ is defined to be $a_1+\cdots a_n$.
\end{enumerate}
\end{Def}

It is known that the set $\mathbf{S}(\{B\})$ is  finite (cf. \cite{IK1}), which explains well-definedness of $\mathrm{GK}(B)$.
We can also see that $\mathrm{GK}(B)$ depends on the equivalence class of  $B$.

\begin{Rmk}\label{rmk:gk}
\begin{enumerate}
    \item
Consider $L \subseteq L' \subseteq V$ such that  $[L':L]=b$, where $b\in \mathbb{Z}_{\geq 0}$.
Here $[L':L]$ is the length of the torsion $\mfo$-module $L'/L$. 
Let $q_{L'}$ be the quadratic form defined on $L'$, whose restriction to $L$ is the same as $q_L$. 
Then we have the followings:
\[
\left\{
  \begin{array}{l l}
 \mathrm{GK}(L) \succeq \mathrm{GK}(L') & \textit{by \cite[Lemma 1.2]{IK1}};\\
|\mathrm{GK}(L')|=|\mathrm{GK}(L)|-2b  & \textit{by \cite[Theorem 0.1]{IK1}};\\
  \textit{$N(L')\subseteq \mfo$ if and only if $\mathrm{GK}(L')\succeq (0, \cdots, 0)$} & \textit{by \cite[Corollary 3.5 and Lemma 3.4]{CY}}.
  \end{array} \right.
\]

\item Let $\overline{L}'\left(=L'\otimes\kappa\right)=\bar{L}'_0\bot \mathrm{Rad}(\overline{L}')$, where $\mathrm{Rad}(\overline{L}')$ is the radical of $\overline{L}'$
so that the restriction of the quadratic form $\bar{q}_{L'}$ on $\bar{L}'_0$ is nonsingular.
We assign the following notion $a^{\pm}$ with an integer $a$ to $L'$ with respect to  $\bar{L}'_0$:
\begin{equation}\label{a}
\left\{
  \begin{array}{l l}
  a^+    & \quad  \textit{if $a=dim~\bar{L}'_0$ is even and $\bar{L}'_0$ is split};\\
  a^-    & \quad  \textit{if $a=dim~\bar{L}'_0$ is even and $\bar{L}'_0$ is nonsplit};\\
    a (=a^+=a^-)   & \quad  \textit{if $a=dim~\bar{L}'_0$ is odd}.
    \end{array} \right.
\end{equation}

Here, $dim~\bar{L}'_0$ is the dimension of $\bar{L}'_0$ as a $\kappa$-vector space.
If $a=0$, then we understand $0^+=0^-$.

\item\cite[Proposition 3.13]{CY}
The integer $a$, which is defined to be the dimension of  $\bar{L}'_0$ as a $\kappa$-vector space,
is the same as the number of $0$'s in $\mathrm{GK}(L')$.
For example, $L'$ is unimodular if and only if $\mathrm{GK}(L')=(0, \cdots, 0)$.

\item Define  $\mathcal{S}_{(L, a^{+}, b)}$ and $\mathcal{S}_{(L, a^{-}, b)}$ as follows:

\begin{equation}\label{sab}
\left\{
  \begin{array}{l}
\mathcal{S}_{(L, a^{+}, b)}=\left\{L' \left(\supseteq L\right) | \mathrm{GK}(L')\succeq (0, \cdots, 0), [L':L]=b,
\textit{$a^+$ is assigned to $L'$}\right\};\\
\mathcal{S}_{(L, a^{-}, b)}=\{L' \left(\supseteq L\right) | \mathrm{GK}(L')\succeq (0, \cdots, 0), [L':L]=b,
\textit{$a^-$ is assigned to $L'$}\}.
    \end{array} \right.
\end{equation}
If $a$ is odd or $0$, then we put $\mathcal{S}_{(L, a^{+}, b)}=\mathcal{S}_{(L, a^{-}, b)}$.
 Note that $\mathcal{S}_{(L, a^{\pm}, b)}$ is empty if $b > \frac{|\mathrm{GK}(L)|}{2}$ by the above (1).
 
\end{enumerate}
\end{Rmk}

\begin{Def}\cite[Definition 3.3 and Lemma 3.2]{CY}
The local density associated to the pair of two quadratic lattices $L$ and $H_k$, denoted by $\alpha(L, H_k)$, is defined as 
 $$\alpha(L, H_k)=\lim_{N\rightarrow \infty} f^{-N (2kn-(n^2+n)/2)}\cdot 
\#\mathrm{O}_{\mfo}(L, H_k)(\mfo/\pi^N \mfo). $$

Here $\mathrm{O}_{\mfo}(L, H_k)$ is the affine scheme defined over $\mfo$ such that $\mathrm{O}_{\mfo}(L, H_k)(R)$, the set of $R$-points of $\mathrm{O}_{\mfo}(L, H_k)$ for any commutative $\mfo$-algebra $R$, is  the set of $R$-linear maps from $L\otimes R$ to $H_k\otimes R$ preserving the associated quadratic forms.  
\end{Def}

The local density $\alpha(L, H_k)$ can be defined as the volume of the $p$-adic manifold  $\mathrm{O}_{\mfo}(L, H_k)(\mfo)$. We  refer to \cite[Section 3.1]{CY} for a detailed explanation. 
The consequence of this description for $\alpha(L, H_k)$ is the following result, which is the starting point of our work.

\begin{Prop}\cite[Theorem 3.14]{CY}\label{prop2.4}
    For a quadratic lattice $L$, let $n_0$ be the number of $0$'s in $\mathrm{GK}(L)=(a_1, \cdots, a_n)$. We have 
\begin{multline*}
\alpha(L, H_k)=\\
(1-f^{-k})\cdot
\sum\limits_{\substack{0\leq b\leq \frac{|\mathrm{GK}(L)|}{2},\\ n_0\leq a\leq n}} \left(\#\mathcal{S}_{(L, a^{\pm}, b)}\cdot
f^{b\cdot(n+1-2k)}\cdot (1+\chi(a^{\pm})f^{n-a/2-k})\prod_{1\leq i < n-a/2}(1-f^{2i-2k})\right),
\end{multline*}
where \[
\chi(a^{\pm})=\left\{
  \begin{array}{l l}
  0    & \quad  \textit{if $a$ is odd};\\
  1    & \quad  \textit{if $a$ is even and $a^+$ is assigned or if $a=0$};\\
   -1   & \quad  \textit{if $a(>0)$ is even and $a^-$ is assigned}.
    \end{array} \right.
\]
Here, if $a$ is odd or $0$, then we ignore one of $\mathcal{S}_{(L, a^{+}, b)}$ or $\mathcal{S}_{(L, a^{-}, b)}$.
If $a$ is even and positive, then we count the summands involving $\mathcal{S}_{(L, a^{+}, b)}$ and $\mathcal{S}_{(L, a^{-}, b)}$ separately.

\end{Prop}

\section{Siegel series in terms of lattice counting}\label{sec:3}
In this section, we will describe the Siegel series as a polynomial $F_L(X)$ in $\mathbb{Z}[X]$ and explain the functional equation in terms of  coefficients of $F_L(X)$.
Then we will describe each coefficient of $F_L(X)$ in terms of $\# \mathcal{S}_{(L,a^\pm,b)}$'s in Theorem \ref{thm:func_eq_coeff}.

Let $B$ be a non-degenerate half-integral symmetrix matrix $B$ over $\mfo$ representing a quadratic lattice $(L, q_L)$. 
We first introduce a few notations, taken from \cite{IK1} and \cite{IK2}, associated with $B$:
\begin{enumerate}
    \item\cite[Introduction]{IK1}  $D_B:=(-4)^{[n/2]}\det B$;
    \item\cite[Introduction]{IK1} 
If $n$ is even, $\mathfrak{D} _B:=$ the discriminant ideal of $F(\sqrt{D_B})/F$;
\item\cite[Introduction]{IK1} $\xi_B:=$ the Kronecker invariant of $B$ defined by
\[
\xi_B=
\begin{cases}
1 & \text{ if $D_B\in F^{\times 2}$,} \\
-1 & \text{ if $F(\sqrt{D_B})/F$ is unramified and $[F(\sqrt{D_B}):F]=2$,} \\
0 & \text{ if $F(\sqrt{D_B})/F$ is ramified;}
\end{cases}
\]
\item\cite[Section 2]{IK2} Put
\[
\mathfrak{e}_B=
\left\{
  \begin{array}{l l}
  \mathrm{ord}(D_B)-\mathrm{ord}(\mathfrak{D}_B)   & \quad  \textit{if $n$ is even};\\
  \mathrm{ord}(D_B)  & \quad  \textit{if $n$ is odd};
      \end{array} \right.
\]
Note that $\mathfrak{e}_B$ is even when $n$ is even. 
\cite[Theorem 0.1]{IK1} yields that
\begin{equation}\label{eq:gke}
|\mathrm{GK}(L)|=\left\{
 \begin{array}{l l}
 \mathfrak{e}_B   & \quad  \textit{if $n$ is odd or if $n$ is even and $\mathrm{ord}(\mathfrak{D}_B)=0$};\\
  \mathfrak{e}_B+1  & \quad  \textit{if $n$ is even and $\mathrm{ord}(\mathfrak{D}_B)>0$}.
      \end{array} \right.
\end{equation}

\item\cite[Introduction]{IK1} If $n$ is odd, then the Clifford invariant of $B$ is defined as follows:
\begin{align*}
\eta_B
:=&
\begin{cases}
1 & \text{ if  $B$ is split over $F$,} \\
-1 & \text{ otherwise.}
\end{cases}
\end{align*}

\item\cite[Section 2]{IK2} $\gamma_{L}(f^{-k}):=$ the rational function with respect to $f^{-k}$ such that
\[
\gamma_{L}(f^{-k})=\left\{
  \begin{array}{l l}
(1-f^{-k})\cdot \left(\prod\limits_{1 \leq i \leq n/2}(1-f^{2i-2k})\right)\cdot (1-\xi_Bf^{n/2-k})^{-1}   & \quad  \textit{if $n$ is even};\\
  (1-f^{-k})\cdot \prod\limits_{1 \leq i \leq (n-1)/2}(1-f^{2i-2k})   & \quad  \textit{if $n$ is odd};
    \end{array} \right.
\]
\end{enumerate}

\begin{Def}\cite[Definition 3.15]{CY}\label{def:siegelseries}
 For a quadratic lattice $L$, the Siegel series is defined to be the polynomial $F_L(X)$ such that 
$$F_L(f^{-k})=\frac{\alpha(L, H_k)}{\gamma_{L}(f^{-k})}.$$
\end{Def}
Note that our definition of the Siegel series is normalized by dividing by $\gamma_{L}(f^{-k})$, compared with that given in \cite[Definition 3.15]{CY}. 
Then Proposition \ref{prop2.4} is reformulated in terms of $F_L(X)$ as follows:

\begin{Prop}\label{flx}

\begin{enumerate}
\item If $n$ is even, then
\begin{multline*}
F_L(X)=\#\mathcal{S}_{(L, n^{\pm}, \frac{|\mathrm{GK}(L)|}{2})}f^{\frac{|\mathrm{GK}(L)|}{2}\cdot(n+1)}X^{|\mathrm{GK}(L)|}+\\
(1-\xi_Bf^{n/2}X) \sum\limits_{\substack{0\leq b\leq \frac{|\mathrm{GK}(L)|}{2},\\n_0\leq a < n}} \left(\#\mathcal{S}_{(L, a^{\pm}, b)}\cdot 
f^{b\cdot(n+1)}X^{2b}\cdot (1+\chi(a^{\pm})f^{n-a/2}X) \cdot \prod_{n/2 < i < n-a/2}(1-f^{2i}X^2)\right).
\end{multline*}
Here, we put $\#\mathcal{S}_{(L, n^{\pm}, \frac{|\mathrm{GK}(L)|}{2})}=0$ if $|\mathrm{GK}(L)|$ is odd.

\item If $n$ is odd, then
\begin{multline*}
F_L(X)=
\sum\limits_{\substack{0\leq b\leq \frac{|\mathrm{GK}(L)|}{2},\\ n_0\leq a\leq n}} \left(\#\mathcal{S}_{(L, a^{\pm}, b)}\cdot 
f^{b\cdot(n+1)}X^{2b}\cdot (1+\chi(a^{\pm})f^{n-a/2}X)\cdot \prod_{n/2 < i < n-a/2}(1-f^{2i}X^2)\right).
\end{multline*}
\end{enumerate}
\end{Prop}

This proposition yields that $F_L(X)$ is in $\mathbb{Z}[X]$.

In order to explain the functional equation of $F_L(X)$, we will introduce 
 $\tilde{F}_L(X):=X^{-\frac{\mathfrak{e}_B}{2}}F_L(f^{-(n+1)/2}X)$, where $X^{1/2}$ is the symbol such that $(X^{1/2})^2=X$ (cf. \cite[Section 2]{IK2}).
Note that $\tilde{F}_L(X) \in \mathbb{Q}[f^{1/2}][X,X^{-1}]$ if $n$ is even, and $\tilde{F}_L(X) \in \mathbb{Q}[X^{1/2},X^{-1/2}]$ if $n$ is odd. 
The functional equation  of $F_L(X)$ or of $\tilde{F}_L(X)$ is described as follows:

\begin{Thm}\cite[Theorem 4.1]{I1} or  \cite[Proposition 2.1]{IK2}\label{fe}
\[
\tilde{F}_L(X^{-1})=\zeta_B\tilde{F}_L(X) ~~~~~ \textit{ equivalently } ~~~~  
F_L(f^{-n-1}X^{-1})=\zeta_B(f^{(n+1)/2}X)^{-\mathfrak{e}_B}F_L(X), 
\]
where $\zeta_B=\eta_B$ or $1$ according as $n$ is odd or even respectively.

\end{Thm}

\begin{Cor}
\begin{enumerate}
    \item 
    The degree of the polynomial $F_L(X)$, which is the same as the degree of the polynomial  $X^{\frac{\mathfrak{e}_B}{2}}\cdot \tilde{F}_L(X)$, is $\mathfrak{e}_B$.
\item The leading coefficient of $F_L(X)$ is $\zeta_B\cdot f^{n(n+1)/2}$.
\end{enumerate}
\end{Cor}
\begin{proof}
By Theorem \ref{fe}, the degree of  $F_L(X)$ is at most $\mathfrak{e}_B$.
Then it suffices to show that $F_L(0)\neq 0$. This follows from  Proposition \ref{flx}, that is, $F_L(0)=\#\mathcal{S}_{(L, n_0^{\pm}, 0)}=1$ since $\mathcal{S}_{(L, n_0^{\pm}, 0)}=\{L\}$.
This also yields the second claim by the functional equation.
\end{proof}

We refer to Equation (\ref{eq:gke}) for the relation between $|\mathrm{GK}(L)|$ and $\mathfrak{e}_B$.
We write
\[
X^{\frac{\mathfrak{e}_B}{2}}\cdot \tilde{F}_L(X)\left(=F_L(f^{-(n+1)/2}X)\right) =\sum\limits_{t=0}^{\mathfrak{e}_B}c_tX^t  ~~~~   \textit{   where $c_t\in \mathbb{Q}[f^{1/2}]$}.
\]
Then the functional equation in Theorem \ref{fe} is rephrased in terms of $c_t$'s as follows:
\begin{Thm}[Rephrasement of Theorem \ref{fe}] \label{thm:fe_coeff}
\[c_t=\zeta_B\cdot  c_{\mathfrak{e}_B-t} ~~~~ \textit{ where $0\leq t \leq \mathfrak{e}_B$}.\]
\end{Thm}

The following theorem describes $c_t$'s in terms of lattice counting $\#\mathcal{S}_{(L,a^\pm,b)}$.

\begin{Thm}\label{thm:func_eq_coeff}
The coefficients $c_t$'s are described as follows:
\begin{enumerate}
    \item 
When $n$ is even,
\begin{enumerate}
\item if $t=2s+1$ is odd, then  
\begin{multline*}
c_{2s+1}=\\
\sum\limits_{\substack{0\leq b\leq s,\\ n_0\leq a < n}} \left(\#\mathcal{S}_{(L, a^{\pm}, b)}
\left(-f^{-1/2}\xi_B+\chi(a^{\pm})f^{\frac{n-a-1}{2}}\right) 
\left(\sum\limits_{\substack{i_1<\cdots < i_{s-b},  \\\frac{n}{2} < i_u< n-\frac{a}{2}}}
(-1)^{s-b} f^{2(i_1+\cdots +i_{s-b})-(n+1)(s-b)}\right)\right);
\end{multline*}
\item if $t=2s$ is even, then 
\begin{multline*}
c_{2s}=\sum\limits_{\substack{0\leq b\leq s;\\ n_0\leq a \leq n }} \left(\#\mathcal{S}_{(L, a^{\pm}, b)}\cdot
\sum\limits_{\substack{i_1<\cdots < i_{s-b},  \\\frac{n}{2} < i_u< n-\frac{a}{2}}}
(-1)^{s-b}\cdot f^{2(i_1+\cdots +i_{s-b})-(n+1)(s-b)}\right)-\\
\sum\limits_{\substack{0\leq b\leq s-1;\\ n_0\leq a < n;\\ a:even}} \xi_B\#\mathcal{S}_{(L, a^{\pm}, b)}
  \chi(a^{\pm})f^{\frac{n-a-2}{2}} 
\left(\sum\limits_{\substack{i_1<\cdots < i_{s-b-1},  \\\frac{n}{2} < i_u< n-\frac{a}{2}}}
(-1)^{s-b-1}\cdot f^{2(i_1+\cdots +i_{s-b-1})-(n+1)(s-b-1)})\right).
\end{multline*}
\end{enumerate}
\item When $n$ is odd, $c_t=$
\[
\left\{
  \begin{array}{l l}
\sum\limits_{\substack{0\leq b\leq s;\\ n_0\leq a< n;\\ a:even}} 
\left(\#\mathcal{S}_{(L, a^{\pm}, b)} \chi(a^{\pm})f^{\frac{n-a-1}{2}}
\left(\sum\limits_{\substack{i_1<\cdots < i_{s-b}; \\\frac{n}{2} < i_u< n-\frac{a}{2}}}
(-1)^{s-b}  f^{2(i_1+\cdots +i_{s-b})-(n+1)(s-b)}\right)\right)  & \quad  \textit{if $t=2s+1$};\\
  \sum\limits_{\substack{0\leq b\leq s;\\ n_0\leq a\leq n}} \left(\#\mathcal{S}_{(L, a^{\pm}, b)}
\left(\sum\limits_{\substack{i_1<\cdots < i_{s-b};  \\\frac{n}{2} < i_u< n-\frac{a}{2}}}
(-1)^{s-b} f^{2(i_1+\cdots +i_{s-b})-(n+1)(s-b)}\right)\right)   & \quad  \textit{if $t=2s$}.
    \end{array} \right.
\]
\end{enumerate}
In the above formulas, the second sum follows the  rules:
\[
\sum\limits_{\substack{i_1<\cdots < i_{s-b}; \\\frac{n}{2} < i_u< n-\frac{a}{2}}}\ast=\left\{
  \begin{array}{l l}
0  & \quad  \textit{if $s>b$ and $\#$the index set $<s-b$};\\
1  & \quad  \textit{if $s=b$};
    \end{array} \right.
\]
\[
\sum\limits_{\substack{i_1<\cdots < i_{s-b-1}; \\\frac{n}{2} < i_u< n-\frac{a}{2}}}\ast=\left\{
  \begin{array}{l l}
0  & \quad  \textit{if $s>b+1$ and $\#$the index set $<s-b-1$};\\
1  & \quad  \textit{if $s=b+1$}.
    \end{array} \right.
\]
\end{Thm}

\begin{proof}
Since $X^{\frac{\mathfrak{e}_B}{2}}\tilde{F}_L(X)=F_L(f^{-(n+1)/2}X)$, Proposition \ref{flx} yields the following description of $X^{\mathfrak{e}_B/2}\tilde{F}_L(X)$:
\begin{enumerate}
\item  If $n$ is even, then
\begin{multline*}
X^{\mathfrak{e}_B/2}\tilde{F}_L(X)=
\#\mathcal{S}_{(L, n^{\pm}, \frac{|\mathrm{GK}(L)|}{2})}X^{|\mathrm{GK}(L)|}+\\
(1-\xi_B f^{-1/2} X)\sum\limits_{\substack{0\leq b\leq \frac{|\mathrm{GK}(L)|}{2},\\ n_0\leq a < n}}
\left(\#\mathcal{S}_{(L, a^{\pm}, b)}\cdot X^{2b}\cdot (1+\chi(a^{\pm})f^{(n-a-1)/2}X) \cdot \prod_{n/2 < i < n-a/2}(1-f^{2i-n-1}X^2)\right).
\end{multline*}
Here, we put
\[
\left\{
  \begin{array}{l}
\textit{$\#\mathcal{S}_{(L, n^{\pm}, \frac{|\mathrm{GK}(L)|}{2})}=0$ if $|\mathrm{GK}(L)|(=\mathfrak{e}_B+1)$ is odd (equivalently $\mathrm{ord}(\mathfrak{D}_B)>0$ by Equation (\ref{eq:gke}))};\\
\textit{$\prod\limits_{n/2 < i < n-a/2}(1-f^{2i-n-1}X^2)=1$ if $\{i|n/2 < i < n-a/2\}=\emptyset$}.
\end{array} \right.
\]


\item If $n$ is odd, then
\begin{multline*}
X^{\mathfrak{e}_B/2}\tilde{F}_L(X)=
\sum\limits_{\substack{0\leq b\leq \frac{|\mathrm{GK}(L)|}{2},\\ n_0\leq a \leq n}} \left(\#\mathcal{S}_{(L, a^{\pm}, b)}
X^{2b}\cdot (1+\chi(a^{\pm})f^{(n-a-1)/2}X)  \prod_{n/2 < i < n-a/2}(1-f^{2i-n-1}X^2)\right).
\end{multline*}
Here $\prod\limits_{n/2 < i < n-a/2}(1-f^{2i-n-1}X^2)=1$ if $\{i|n/2 < i < n-a/2\}=\emptyset$.
\end{enumerate}

The desired formulas for $c_t$'s are  obtained by analyzing the above formulas. 
\end{proof}

In the following three corollaries, we will describe $c_t$'s explicitly in certain situations.

\begin{Cor}\label{cor:coeffofct}
The coefficients $c_t$ with $0\leq t\leq 3$ are described as follows:
\begin{enumerate}
    \item When $n$ is even,
\begin{enumerate}
    \item $c_0=1$;
    \item $c_1=-f^{-1/2}\xi_B+\chi(n_0^{\pm})f^{(n-n_0-1)/2}$;

\item  
$c_2=\#\mathcal{S}_{(L, n_0^{\pm}, 1)}+\#\mathcal{S}_{(L, (n_0+1)^{\pm}, 1)}+\#\mathcal{S}_{(L, (n_0+2)^{\pm}, 1)}-\xi_B\chi(n_0^{\pm})f^{\frac{n-n_0-2}{2}}-\sum\limits_{\frac{n}{2}<i<n-\frac{n_0}{2}}f^{2i-n-1}$;

\item 
$c_3$ is described as follows:
\begin{enumerate}
    \item if $n_0$ is odd, then 
$$\#\mathcal{S}_{(L, (n_0+1)^{\pm}, 1)}\chi((n_0+1)^{\pm})f^{\frac{n-n_0-2}{2}}
-f^{\frac{-1}{2}}\xi_B\left(\sum\limits_{i=0}^{2}\#\mathcal{S}_{(L, (n_0+i)^{\pm}, 1)}
-\sum\limits_{\frac{n}{2}<i<n-\frac{n_0}{2}}f^{2i-n-1}\right);$$ 
\item If $n_0$ is even, then 
\begin{multline*}
    \left(\chi(n_0^{\pm})\left(f\#\mathcal{S}_{(L, n_0^{\pm}, 1)}-\sum\limits_{\frac{n}{2}<i<n-\frac{n_0}{2}}f^{2i-n}\right)
+\#\mathcal{S}_{(L, (n_0+2)^{\pm}, 1)}\chi((n_0+2)^{\pm})\right)f^{\frac{n-n_0-3}{2}}
-\\
f^{\frac{-1}{2}}\xi_B\left(\sum\limits_{i=0}^{2}\#\mathcal{S}_{(L, (n_0+i)^{\pm}, 1)}-\sum\limits_{\frac{n}{2}<i<n-\frac{n_0}{2}}f^{2i-n-1}\right).
\end{multline*}
\end{enumerate}

\end{enumerate}    

    \item When $n$ is odd, 
\begin{enumerate}
    \item $c_0=1$;

    \item $c_1=\chi(n_0^{\pm})f^{(n-n_0-1)/2}$;

\item  
$c_2=\#\mathcal{S}_{(L, n_0^{\pm}, 1)}+\#\mathcal{S}_{(L, (n_0+1)^{\pm}, 1)}+\#\mathcal{S}_{(L, (n_0+2)^{\pm}, 1)}-\sum\limits_{\frac{n}{2}<i<n-\frac{n_0}{2}}f^{2i-n-1}$;

\item 
$c_3$ is described as follows:
\[
\left\{
\begin{array}{l l}
\#\mathcal{S}_{(L, (n_0+1)^{\pm}, 1)}\chi((n_0+1)^{\pm})f^{\frac{n-n_0-2}{2}}
  & \quad  \textit{if $n_0$ is odd};\\
    \left(\chi(n_0^{\pm})\left(f\#\mathcal{S}_{(L, n_0^{\pm}, 1)}-\sum\limits_{\frac{n}{2}<i<n-\frac{n_0}{2}}f^{2i-n}\right)
+\#\mathcal{S}_{(L, (n_0+2)^{\pm}, 1)}\chi((n_0+2)^{\pm})\right)f^{\frac{n-n_0-3}{2}}
  & \quad  \textit{if $n_0$ is even}.
    \end{array} \right.
\]

\end{enumerate}

\end{enumerate}
Here $\#\mathcal{S}_{(L, a^{\pm}, b)}$ is interpreted as explained in Proposition \ref{prop2.4}.     
\end{Cor}


\begin{remark}
    In general, it is technically difficult to simplify the coefficients $c_t$ for $t \geq 4$ as in Corollary \ref{cor:coeffofct}, due to the complexity of the sum $\sum\limits_{\substack{i_1<\cdots < i_{s-b}; \\\frac{n}{2} < i_u< n-\frac{a}{2}}}\ast$ in the formulas of  Theorem \ref{thm:func_eq_coeff}.
    In order to treat this sum, we restrict to the cases that $n_0$ is either $n-1$ or $n-2$ in the following Corollaries \ref{cor:coeff_n0=n-1}-\ref{cor:coeff_n0=n-2} respectively.
    This makes the index set of the sum  simpler since an integer $a$ should satisfy $n_0 \leq a \leq n$ (cf. Theorem \ref{thm:func_eq_coeff}).
\end{remark}

\begin{Cor}\label{cor:coeff_n0=n-1}
Suppose that  $n_0=n-1$. 
The coefficients $c_t$'s  are described as follows:
\begin{enumerate}
    \item When $n$ is even,
\[
\left\{
  \begin{array}{l l}
c_{2s+1}=
-f^{-1/2}\xi_B\#\mathcal{S}_{(L, n-1, s)};\\
c_{2s}=
\#\mathcal{S}_{(L, n-1, s)}
  & \quad  \textit{if $2s<\mathfrak{e}_B$};\\
c_{\mathfrak{e}_B}=
\begin{cases}
    \#\mathcal{S}_{(L, n^\pm, \frac{\mathfrak{e}_B}{2})} &\textit{if } |\mathrm{GK}(L)| \textit{ is even} \\
    \#\mathcal{S}_{(L,(n-1)^\pm,\frac{\mathfrak{e}_B}{2}-1)} &\textit{if }|\mathrm{GK}(L)| \textit{ is odd}
\end{cases}
.
    \end{array} \right.
\]

    \item When $n$ is odd, 
\[
\left\{
  \begin{array}{l l}
c_{2s+1}=\chi((n-1)^\pm)\#\mathcal{S}_{(L, (n-1)^\pm, s)};\\
c_{2s}=\#\mathcal{S}_{(L, (n-1)^\pm, s)}
  & \quad  \textit{if $2s<\mathfrak{e}_B$};\\
c_{\mathfrak{e}_B}=
\begin{cases}
    \#\mathcal{S}_{(L, n^\pm, \frac{\mathfrak{e}_B}{2})} &\textit{if } \mathfrak{e}_B \textit{ is even} \\
    \chi((n-1)^\pm)\#\mathcal{S}_{(L,(n-1)^\pm,\frac{(\mathfrak{e}_B-1)}{2})} &\textit{if }\mathfrak{e}_B \textit{ is odd}
\end{cases}
    \end{array} \right.
\]

\end{enumerate}
  Here $\#\mathcal{S}_{(L, a^{\pm}, b)}$ is interpreted as explained in Proposition \ref{prop2.4}.   
\end{Cor}

\begin{Cor}\label{cor:coeff_n0=n-2}
Suppose that $n_0=n-2$. 
The coefficients $c_t$'s  are described as follows:
\begin{enumerate}
    \item When $n$ is even,
    \[
\left\{
  \begin{array}{l l}
c_{2s+1}= f^{1/2} \chi((n-2)^\pm)\#\mathcal{S}_{(L, (n-2)^\pm, s)} - \xi_Bf^{-1/2}\left(\#\mathcal{S}_{(L,(n-2)^\pm, s)}  + \#\mathcal{S}_{(L, n-1, s)}\right);\\
c_{2s}=
\#\mathcal{S}_{(L, (n-2)^\pm, s)}+\#\mathcal{S}_{(L, n-1, s)} - 
\xi_B\chi((n-2)^\pm) \#\mathcal{S}_{(L, (n-2)^\pm, s-1)}
  & \quad  \textit{if $2s<\mathfrak{e}_B$};\\
c_{\mathfrak{e}_B}=
\#\mathcal{S}_{(L, n^\pm, \frac{\mathfrak{e}_B}{2})}-
\xi_B\chi((n-2)^\pm)\#\mathcal{S}_{(L, (n-2)^\pm, \frac{\mathfrak{e}_B}{2}-1)}.
    \end{array} \right.
\]

    \item When $n$ is odd, 
   \[
\left\{
  \begin{array}{l l}
c_{2s+1}=\#\chi((n-1)^{\pm})\mathcal{S}_{(L, (n-1)^\pm, s)};\\
c_{2s}=
\#\mathcal{S}_{(L, (n-1)^\pm, s)}+\#\mathcal{S}_{(L, n-2, s)}- \#\mathcal{S}_{(L, n-2, s-1)}
  & \quad  \textit{if $2s<\mathfrak{e}_B$};\\
c_{\mathfrak{e}_B}=
\begin{cases}
  \#\mathcal{S}_{(L, n, \frac{\mathfrak{e}_B}{2})}-
\#\mathcal{S}_{(L, n-2, \frac{\mathfrak{e}_B}{2}-1)} &\textit{if } |\mathrm{GK}| \textit{ is even}; \\
  \chi((n-1)^\pm)\#\mathcal{S}_{(L, (n-1)^\pm, \frac{\mathfrak{e}_B-1}{2})} &\textit{if } |\mathrm{GK}| \textit{ is odd}.  
\end{cases}

    \end{array} \right.
\]

\end{enumerate}
      Here $\#\mathcal{S}_{(L, a^{\pm}, b)}$ is interpreted as explained in Proposition \ref{prop2.4}.
\end{Cor}

\section{Explicit values of coefficients when  $p$ is odd and $n_0 \geq n-2$}\label{sec:4}
Since Theorem \ref{thm:func_eq_coeff} describes the coefficients $c_t$ of $\tilde{F}_L(X)$ in terms of lattice counting involving $\#\mathcal{S}_{(L, a^{\pm}, b)}$'s, explicit calculation of $\#\mathcal{S}_{(L, a^{\pm}, b)}$'s induces explicit values of $c_t$'s and the functional equation (cf. Theorem \ref{thm:fe_coeff}).

It seems to us that this calculation is technical in general.
In this section, we will compute them (cf. Proposition \ref{prop:Sab}) and then provide the explicit values of $c_t$'s (cf. Theorem \ref{thm45}) when  $p$ is odd and $n_0 \geq n-2$. Explicit values of $c_t$'s directly yield the functional equation of $\tilde{F}_L(X)$ (Corollary \ref{cor:functionaleq}).
From now on, we assume that   $p$ is odd and $n_0 \geq n-2$.

\subsection{Formula for  $\#\mathcal{S}_{(L,a^\pm,b)}$}
As in the previous section, let $B$ be a non-degenerate half-integral symmetrix matrix $B$ over $\mfo$ representing a quadratic lattice $(L, q_L)$. 
Since $p$ is odd and $n_0 \geq n-2$, we may choose  an $\mathfrak{o}$-basis of $L$ such that $B$ is a diagonal matrix described as follows:
\begin{equation}\label{eq:form_B}
B = 
\begin{pmatrix}
    u_1 &\cdots &0 &0 &0 \\
    \vdots &\ddots &\vdots &\vdots &\vdots \\
    0 &\cdots &u_{n-2} &0 &0 \\
    0 &\cdots &0 &\pi^{d_1} v_1 &0 \\
    0 &\cdots &0 &0 &\pi^{d_2} v_2 
\end{pmatrix}
\end{equation}
where $u_i, v_j \in \mathfrak{o}^\times$ and $d_1, d_2$ are non-negative integers with $d_1 \leq d_2$.
Then 
\begin{equation}\label{eq:gk12}
    \mathrm{GK}(L)=\left(0, \cdots, 0, d_1, d_2\right)
\end{equation}
since $p$ is odd (cf. \cite[Remark 1.1]{IK1}). 
We denote by $\overline{u}_i$ the reduction of $u_i$ modulo $\pi$ and the same for $\overline{v}_j$.
Under this basis, we identify $\mathrm{Aut}_F(V) = \mathrm{GL}_n(F)$ and $\mathrm{Aut}_{\mfo}(L) = \mathrm{GL}_n(\mathfrak{o})$.
Let $g$ be a coset in  $\mathrm{GL}_n(F)/\mathrm{GL}_n(\mathfrak{o})$.
There is a well-defined bijection 
\begin{equation}\label{eq:bij_cosets_lattices}
    \mathrm{GL}_n(F)/\mathrm{GL}_n(\mathfrak{o}) \to \{\textit{lattices } L' \subseteq V\}, \quad g \mapsto gL.
\end{equation}
Thus in order to count $\mathcal{S}_{(L,a^\pm,b)}$, we will describe explicit conditions for $g \in \mathrm{GL}_n(F)/\mathrm{GL}_n(\mathfrak{o})$ such that $gL \in \mathcal{S}_{(L,a^\pm,b)}$.
For  $M$ in $\mathrm{M}_n(\mathfrak{o})$, we denote by $\overline{M}\left(\in \mathrm{M}_n(\kappa)\right)$ the reduction of  $M$ modulo $\pi$.


\begin{Lem}\label{lem:rep_g_S_ab}
Let $Id_{n-2}$ be the identity matrix of size $n-2$.
\begin{enumerate}
    \item
 $gL \in \mathcal{S}_{(L,a^\pm,b)}$ if and only if $g$ is represented by an $n\times n$ matrix
   \[
    X = 
    \begin{pmatrix}
    Id_{n-2}  &0 &0 \\
   0 &\pi^{-l} &\pi^{-l}x \\
   0 &0 &\pi^{-(b-l)}
    \end{pmatrix}\in \mathrm{GL}_n(F) \textit{ with $0 \leq l \leq b$ such that } \pi^{b-l} x \in \mathfrak{o} \textit{ and such that}
   \]  
 \[ A:=
  \begin{pmatrix}
          \pi^{d_1-2l} v_1 &\pi^{d_1-2l} v_1 x \\
        \pi^{d_1-2l} v_1 x &\pi^{d_1-2l} v_1 x^2 + \pi^{d_2-2(b-l)} v_2
      \end{pmatrix}\in \mathrm{M}_2(\mfo) \textit{ with  $rank(\overline{A})=a - (n-2)$.}
  \] 
  \item    The image of $x$ in $F/\mathfrak{o}$ and the integer $l$ defined in the above (1) are independent of the choice of a representative $X$ of $g$. Furthermore, these two  determine $g$ uniquely.

      \item    Suppose that $gL \in \mathcal{S}_{(L,a^\pm,b)}$ and that $a$ is even.
 Then
 \[   gL \in \mathcal{S}_{(L,a^+,b)} \textit{ if and only if }
        \begin{cases}
        (-1)^{\frac{n-2}{2}} {u}_1 \cdots {u}_{n-2} \in  (\mathfrak{o}^\times)^2 &\textit{if } a = n-2; \\
        (-1)^{\frac{n-1}{2}}{u}_1 \cdots {u}_{n-2} \cdot a_0 \in (\mathfrak{o}^\times)^2 &\textit{if } a = n-1; \\
        (-1)^{\frac{n}{2}}{u}_1 \cdots {u}_{n-2} \cdot {v}_1 {v_2} \in (\mathfrak{o}^\times)^2 &\textit{if } a = n.
      \end{cases}
   \] 
   Here  $a_0$ is the $(1,1)$-entry of $A$ if it is a unit, and the  $(2,2)$-entry of $A$ otherwise.

This means that  one of $\mathcal{S}_{(L,a^+,b)}$ or $\mathcal{S}_{(L,a^-,b)}$ is empty if $a=n-2$ or $a=n$.

  \end{enumerate}
    
\end{Lem}
\begin{proof}
We write $L' := gL$ for $g\in \mathrm{GL}_n(F)/\mathrm{GL}_n(\mathfrak{o})$.
\begin{enumerate}
    \item
    By Remark \ref{rmk:gk}.(4), $L'$ is contained in $\mathcal{S}_{(L,a^\pm,b)}$ if and only if $L \subseteq L'$, $\mathrm{GK}(L') \succeq (0,\cdots,0)$, $a = \dim\overline{L}'_0$, and $[L':L] = b$.
This is equivalent that $X^{-1}, {}^t X B X \in \mathrm{GL}_n(\mfo)$, $rank(\overline{{}^t X B X})=a$, and $|\det X| = f^{b}$.
    Note that all these conditions are independent of the choice of a representative $X$ of $g$.
    
 Suppose that $gL \in \mathcal{S}_{(L,a^\pm,b)}$.
Since $\mathrm{GL}_n(F)=T\cdot \mathrm{GL}_n(\mathfrak{o})$ with $T$ the group of upper triangular matrices by Iwasawa decomposition, $g$ is represented by an upper triangular matrix    
 $X = 
    \begin{pmatrix}
        \pi^{-l_1} &g_{12} &\cdots &g_{1n} \\
        0 &\pi^{-l_2} &\cdots &g_{2n} \\
        \vdots &\ddots &\ddots  &\vdots \\
        0 &\cdots &0  &\pi^{-l_n}
    \end{pmatrix}$,
    where $l_i\in\mathbb{Z}$ and $g_{ij}\in F$.
    Since $X^{-1} \in\mathrm{GL}_n(\mfo)$, we have that $l_i\geq 0$.
    
We claim that $l_i = 0$ and $g_{ij} \in \mathfrak{o}$ for $1 \leq i \leq n-2$.
    Since ${}^t X B X \in \mathrm{GL}_n(\mfo)$, each entry yields the followings:
    \begin{equation}\label{eq:coeff_tgBg}
    \begin{cases}
         \pi^{-2l_i} u_i + \sum\limits_{k = 1}^{i-1} g_{ki}^2 u_k \in \mathfrak{o} & \textit{by the $(i,i)$-entry for $1 \leq i \leq n-2$};\\
         \pi^{-l_i} g_{ij} u_i + \sum\limits_{k = 1}^{i-1} g_{ki} g_{kj} u_k \in \mathfrak{o} & \textit{by the $(i,j)$-entry  with $i<j$ for $1 \leq i \leq n-2$}.
    \end{cases}    
    \end{equation}
Putting   $i = 1$, the above  yields that $l_1 = 0$ and $g_{1j} \in \mathfrak{o}$.
    We fix $1<i \leq n-2$.
    If $l_k = 0$ and $g_{kj} \in \mathfrak{o}$ for $1 \leq k \leq i-1$, then the above yields that $l_i = 0$ and $g_{ij} \in \mathfrak{o}$.
    This proves the claim inductively.

    Since the first $(n-2)$-diagonal entries of the above $X$ are $1$, by using column additions we may and do assume that $g_{ij} = 0$ for $1\leq i \leq n-2$.
We write $l := l_{n-1}$ and $\pi^{-l}x := g_{(n-1)n}$.
The condition $|\det X| = f^{b}$ induces that  $l_n = b-l$.
Thus we have a matrix $X$ of the form claimed in the statement of (1) and   thus $X^{-1}=
       \begin{pmatrix}
    Id_{n-2}  &0 &0 \\
0 &\pi^{l} &-\pi^{b-l} x \\
0 &0  &\pi^{b-l}    \end{pmatrix}$. 
Since $X^{-1} \in\mathrm{GL}_n(\mfo)$, we have that   $0 \leq l \leq b$
and  that $\pi^{b-l} x \in \mathfrak{o}$.

To prove properties on $A$, we describe ${}^t X B X$ as follows:
\[  
    {}^t X B X =
    \begin{pmatrix}
    u_1  &\cdots &0 &0 &0 \\
    \vdots  &\ddots &\vdots &\vdots &\vdots \\
    0  &\cdots &u_{n-2} &0 &0 \\
    0  &\cdots &0 &\pi^{d_1-2l}v_1 &\pi^{d_1-2l} v_1 x \\
    0  &\cdots &0 &\pi^{d_1-2l} v_1 x &\pi^{d_1-2l} v_1 x^2 + \pi^{d_2-2(b-l)} v_2
    \end{pmatrix}.  
   \]  
The matrix $A$ in the claim  is  the right lower corner of size $2\times 2$, and thus has entries in $\mfo$ as  ${}^t X B X$ does.
The condition $rank(\overline{{}^t X B X})=a$ yields that $rank(\overline{A})=a-(n-2)$. 

   The converse follows directly from the above descriptions of $X^{-1}$ and ${}^t X B X$.
   \\
    
\item 
Let $X$ and $X'$ be $n\times n$ matrices as described in the claim (1) with $(n-2,n)$-entries $\pi^{-l}x$ and $\pi^{-l'}x'$ respectively.
Then a direct matrix computation yields that $X'^{-1} X \in \mathrm{GL}_n(\mathfrak{o})$ if and only if $l = l'$ and both $x$ and $x'$ have the same image in $F/\mathfrak{o}$.
\\





    \item 
   Suppose that $L' \in \mathcal{S}_{(L,a^\pm,b)}$ and that $a$ is even. 
    By Remark \ref{rmk:gk}.(2), let  $\overline{q}_{\bar{L}'_0}$ be the restriction of the quadratic form $\bar{q}_{L'}$  on $\overline{L}'_0$, where  $\overline{L}'=\bar{L}'_0\bot \mathrm{Rad}(\overline{L}')$. 
    Then $L' \in \mathcal{S}_{(L,a^+,b)}$ if and only if $(\overline{L}'_0,\overline{q}_{\bar{L}'_0})$ is split.
    By \cite[Proposition 5.5]{Cas11}, this is equivalent  that the discriminant of $\overline{q}_{\bar{L}'_0}$ multiplied with $(-1)^{\frac{a}{2}}$ is a square in $\kappa^{\times}$. 

We choose a representative $X$ of $g$ as in the claim (1).
    Then $\overline{q}_{L'}$ is represented by the reduction of ${}^t X B X$ modulo $\pi$, which is described in the above proof of the claim (1). 
    The discriminant of $\overline{q}_{L_0'}$ is computed case by case as follows:
    \begin{itemize}
        \item If $a = n-2$, then   the discriminant of $\overline{q}_{L_0}$ is $\overline{u}_1 \cdots \overline{u}_{n-2}$ since $\overline{A} = 0$.
        \item If $a = n$, then  the discriminant of $\overline{q}_{L_0}$ is $\overline{u}_1 \cdots \overline{u}_{n-2}\cdot \overline{v}_1 \overline{v}_2$ since $\overline{L}' = \overline{L}'_0$.
    
        \item If $a = n-1$, then the rank of $\overline{A}$ is 1.
        Since $\overline{A}$ is a symmetric matrix, either the $(1,1)$-entry or the $(2,2)$-entry of $A$ is a unit.
Suppose that the $(1,1)$-entry of $A$ is a unit  and denote it by $w_1$.
Then $\overline{A}$ is of the form $\begin{pmatrix}
            \overline{w}_1 &c\overline{w}_1 \\
            c \overline{w}_1 &c^2\overline{w}_1
        \end{pmatrix}$
        for some $c \in \kappa$.
        For $Y = \begin{pmatrix}
            1 &-c \\
            0 &1
        \end{pmatrix}$, we have that $Y^t \overline{A} Y = \begin{pmatrix}
            \overline{w}_1 &0 \\
            0 & 0
        \end{pmatrix}$.
        Therefore, the discriminant of $\overline{q}_{L'_0}$ is $\overline{u}_1 \cdots \overline{u}_{n-2} \overline{w}_1$.

        If the $(1,1)$-entry of $A$ is not a unit, then the $(2,2)$-entry of $A$ is a unit, which is denoted by  $w_2$.
        In this case, $\overline{A} = \begin{pmatrix}
            0 & 0\\
            0 &\overline{w}_2
        \end{pmatrix}$
so that
    the discriminant of $\overline{q}_{L'_0}$ is $\overline{u}_1 \cdots \overline{u}_{n-2} \overline{w}_2$.
    \end{itemize}
  Note that an element of $\mathfrak{o}^\times$ is a square if and only if its reduction modulo $\pi$ is a square in $\kappa^\times$ since $p$ is odd.
   This completes the proof. 
    \end{enumerate}
\end{proof}

In order to compute $c_t$'s in Corollaries \ref{cor:coeff_n0=n-1}-\ref{cor:coeff_n0=n-2} explicitly, we need to calculate $\# \mathcal{S}_{(L,a^\pm,b)}$ for $n-2 \leq a \leq n$, $ \chi((n-2)^\pm) \#\mathcal{S}_{(L,(n-2)^\pm,b)}$, and $\chi((n-1)^\pm) \#\mathcal{S}_{(L,(n-1)^\pm,b)}$.
We refer to Section  \ref{sec:not} for notations $\delta(-)$, $(-,-)$, $\lceil - \rceil$, and $\lfloor - \rfloor$ which will appear frequently in the following proposition.

\begin{Prop}\label{prop:Sab}
    
    \begin{enumerate}
            \item We have the following formulas for $\# \mathcal{S}_{(L,a^\pm,b)}$:
            {\small
        \begin{equation*}
      \# \mathcal{S}_{(L,(n-2)^\pm,b)}
      =
      \begin{cases}
        \sum_{i = 0}^b f^i &\textit{if } 0\leq b<\frac{d_1}{2}; \\
        \sum_{i = 0}^{\lfloor (d_1-1)/2 \rfloor} f^i
        &\textit{if } \frac{d_1}{2} \leq b < \frac{d_2}{2}; \\
        \sum_{i = 0}^{\lfloor (d_1-1)/2\rfloor +\lfloor (d_2-1)/2\rfloor-b} f^i 
        &\textit{if } \frac{d_2}{2} \leq b<\frac{d_1+d_2}{2} \textit{ and } \delta(d_1+d_2)\neq(-v_1v_2,\pi); \\
        (2b-d_2+1) f^{\frac{d_1+d_2}{2}-b-1}+\sum_{i = 0}^{\frac{d_1+d_2}{2}-b-1} f^i
        &\textit{if } \frac{d_2}{2} \leq b<\frac{d_1+d_2}{2} \textit{ and }
         \delta(d_1+d_2)=(-v_1v_2,\pi); \\
        0 &\textit{if } \frac{d_1+d_2}{2} \leq b,
      \end{cases} \\
      \end{equation*}
      
      \begin{equation*}
      \# \mathcal{S}_{(L,(n-1)^\pm,b)} = 
      \begin{cases}
         0 &\textit{if } 0 \leq b <\frac{d_1}{2} \textit{ or } \frac{d_1+d_2}{2} \leq b; \\
         \delta(d_1) f^{\frac{d_1}{2}} &\textit{if } \frac{d_1}{2} \leq b < \frac{d_2}{2}; \\
         \delta(d_1) f^{\frac{d_1}{2}+\lfloor \frac{d_2-1}{2}\rfloor-b} + \delta(d_2)f^{\lfloor \frac{d_1+d_2}{2} \rfloor-b}
         &\textit{if } \frac{d_2}{2} \leq b<\frac{d_1+d_2}{2} \textit{ and }  \delta(d_1+d_2)\neq (-v_1v_2,\pi) \\
         (2b-d_2+1)(f-1)f^{\frac{d_1+d_2}{2}-b-1}
        &\textit{if } \frac{d_2}{2} \leq b<\frac{d_1+d_2}{2} 
         \textit{ and } \delta(d_1+d_2)=(-v_1v_2,\pi).
       \end{cases}
    \end{equation*}
    \begin{equation*}
        \# \mathcal{S}_{(L,n^\pm,b)}=
        \begin{cases}
            0 &\textit{if } b \neq\frac{d_1+d_2}{2}; \\
            d_1 + 1 &\textit{if } b=\frac{d_1+d_2}{2} \textit{ and } 
            \delta(d_1+d_2) = (-v_1v_2,\pi); \\
            \delta(d_1)\delta(d_2) &\textit{if } b=\frac{d_1+d_2}{2} \textit{ and } 
            \delta(d_1+d_2) \neq (-v_1v_2,\pi).
        \end{cases}
    \end{equation*}
    }
    Here $\delta(d_1+d_2) = (-v_1v_2,\pi)$ if and only if $d_1+d_2$ is even and $-v_1v_2\in (\mathfrak{o}^\times)^2$ in Section \ref{sec:not}.

  \item If $n$ is even, then we have that
  {\small
    \begin{equation*}
             \chi((n-2)^\pm)\#\mathcal{S}_{(L,(n-2)^\pm,b)} =
             \begin{cases}
        ((-1)^{\frac{n-2}{2}}{u}_1 \cdots {u}_{n-2},\pi) \cdot \#\mathcal{S}_{(L,(n-2)^\pm,b)} &\textit{if } 0 \leq b < \frac{d_1+d_2}{2};\\
        0 &\textit{if } \frac{d_1+d_2}{2} \leq b.
        \end{cases}
    \end{equation*}
    }
               Here $\#\mathcal{S}_{(L,(n-2)^\pm,b)}$ is computed in (1).
 
  \item If $n$ is odd, then we have that
  {\small
  \begin{equation*}
      \chi((n-1)^\pm)\#\mathcal{S}_{(L,(n-1)^\pm,b)} = 
      \begin{cases}
         0 &\textit{if } 0 \leq b <\frac{d_1}{2} \textit{ or } \frac{d_1+d_2}{2} \leq b; \\
         ((-1)^{\frac{n-1}{2}}{u}_1 \cdots {u}_{n-2} v_1,\pi) \cdot \delta(d_1)f^{\frac{d_1}{2}} &\textit{if } \frac{d_1}{2} \leq b < \frac{d_2}{2}; \\
        ((-1)^{\frac{n-1}{2}}{u}_1 \cdots {u}_{n-2}v_1,\pi) \cdot f^{\frac{d_1+d_2-1}{2}-b} 
         &\textit{if } \frac{d_2}{2} \leq b<\frac{d_1+d_2}{2},  \delta(d_1)=1,\delta(d_2)=0; \\
         ((-1)^{\frac{n-1}{2}}{u}_1 \cdots {u}_{n-2}v_2,\pi) \cdot f^{\frac{d_1+d_2-1}{2}-b} 
         &\textit{if } \frac{d_2}{2} \leq b<\frac{d_1+d_2}{2},\delta(d_1)=0,\delta(d_2)=1;\\
            0 & \textit{if $\frac{d_2}{2} \leq b<\frac{d_1+d_2}{2}$, $\delta(d_1+d_2)=1$}.
      \end{cases} 
  \end{equation*}
  }
  \end{enumerate}    
  \end{Prop}

\begin{remark}{\textbf{Strategy of the proof.}}\label{rmk:4.3}
By  Lemma \ref{lem:rep_g_S_ab}.(1)-(2), the computation of $\# \mathcal{S}_{(L,a^\pm,b)}$
    is equivalent to count the number of integers $0 \leq l \leq b$ and $x \in F$ up to addition by elements of  $\mathfrak{o}$, satisfying certain conditions listed in Lemma \ref{lem:rep_g_S_ab}.(1) involving the matrix 
    \begin{equation}\label{eq:form_A}
        A:=
  \begin{pmatrix}
          \pi^{d_1-2l} v_1 &\pi^{d_1-2l} v_1 x \\
        \pi^{d_1-2l} v_1 x &\pi^{d_1-2l} v_1 x^2 + \pi^{d_2-2(b-l)} v_2
      \end{pmatrix}.
    \end{equation}

    Let $S$ be a set of representatives of $\kappa$ in $\mathfrak{o}$ such that 
    $0, \pm 1\in S$. 
For a fixed unit $w \in \mathfrak{o}^\times$,  $x \in F$ up to addition by elements of  $\mathfrak{o}$ is uniquely expressed as follows:
    \begin{equation}\label{eq:form_x}
    x = w\left(\sum_{i = 1}^\infty s_i \pi^{-i}\right) \textit{   with   } s_i \in S   \textit{ such that all but finitely many $s_i$'s are $0$.}
    \end{equation}
  
    Later for each case, we will choose a suitable $w$ to make the computation easier.
    We denote by $\overline{s}_i \in \kappa$ the reduction of $s_i\in \mfo$ modulo $\pi$.
Then our steps go as follows:

    \begin{enumerate}[label=Step \arabic*:]
        \item The computation of $\# \mathcal{S}_{(L,a^\pm,b)}$:
    
    The condition that $0 \leq l \leq b$, $\pi^{b-l} x \in \mathfrak{o}$, and $A \in \mathrm{M}_2(\mathfrak{o})$ in Lemma \ref{lem:rep_g_S_ab}.(1) is equivalent to the following conditions involving $l$ and $s_i$'s:
    \begin{equation}\label{eq:Sab}
    \begin{cases}
      (1): 0 \leq l \leq \min(d_1/2,b);  \\
      (2): s_i = 0 \textit{ for } i > \min(b-l,d_1-2l); \\
      (3): \pi^{d_1-2l} v_1w^2 \left(\sum\limits_{i = 1}^\infty s_i \pi^{-i}\right)^2 + \pi^{d_2-2(b-l)} v_2  \in \mathfrak{o}.
    \end{cases}
  \end{equation}
  Here the condition for the rank of $\overline{A}$ in Lemma \ref{lem:rep_g_S_ab}.(1) is not imposed yet, and thus the number of $l$ and $s_i$'s satisfying Equation (\ref{eq:Sab}) equals $\sum\limits_{a = n-2}^n\# \mathcal{S}_{(L,a^\pm,b)}$.

 To treat  the rank of $\overline{A}$, we observe that 
 $|\det A| = |\det{{}^t X B X}|= f^{2b - d_1 - d_2}$ (cf. the proof of Lemma \ref{lem:rep_g_S_ab}.(1)). Then   the conditions that $A \in \mathrm{M}_2(\mathfrak{o})$  and $rank (\overline{A}) = a-(n-2)$ in Lemma \ref{lem:rep_g_S_ab}.(1) yield that
    \begin{equation}\label{equation:bd12}
            \begin{cases}
       \mathcal{S}_{(L,a^{\pm},b)} = \emptyset \textit{ for any } a  &\textit{if } b > (d_1+d_2)/2; \\
       \mathcal{S}_{(L,n^\pm,b)} = \emptyset &\textit{if }b \neq (d_1+d_2)/2; \\
       \mathcal{S}_{(L,a^\pm,(d_1+d_2)/2)} = \emptyset &\textit{if } a \neq n.
    \end{cases}
        \end{equation}

This directly implies that the number of $l$ and $s_i$'s satisfying Equation (\ref{eq:Sab}) $=$
    \[
  \sum\limits_{a = n-2}^n\# \mathcal{S}_{(L,a^\pm,b)}=
    \begin{cases}
        \# \mathcal{S}_{(L,(n-2)^\pm,b)} + \# \mathcal{S}_{(L,(n-1)^\pm,b)} &\textit{if } b<\frac{d_1+d_2}{2}; \\
        \# \mathcal{S}_{(L,n^\pm,b)} &\textit{if } b = \frac{d_1+d_2}{2}.
    \end{cases}
    \]

    The remaining is to compute $\#\mathcal{S}_{(L,(n-2)^\pm,b)}$
    when $b<(d_1+d_2)/2$. 
By using the description of $A$ given in Equation (\ref{eq:form_A}),  the condition that $ \overline{A} = 0$ under the restriction of Equation (\ref{eq:Sab}) is equivalent to   
   the following three conditions:
  \begin{equation}\label{eq:Sab2}
  (1') : l<d_1/2; \quad
  (2') : s_{d_1-2l} = 0; \quad
  (3') : \pi^{d_1-2l} v_1w^2 \left(\sum_{i = 1}^\infty s_i \pi^{-i}\right)^2 + \pi^{d_2-2(b-l)} v_2  \in \pi\mathfrak{o}.
  \end{equation}
  Thus  $\# \mathcal{S}_{(L,(n-2)^\pm,b)}=$  the number of $l$ and $s_i$'s satisfying Equations (\ref{eq:Sab})-(\ref{eq:Sab2}).
\\

    \item The computation of $\chi(a^\pm)\# \mathcal{S}_{(L,a^\pm,b)}$ for $a (<n)$ even:

    If $n$ is even, then
Lemma \ref{lem:rep_g_S_ab}.(3) yields that  one of $\mathcal{S}_{(L,(n-2)^+,b)}$ or $\mathcal{S}_{(L,(n-2)^-,b)}$ is empty, and that $\#\mathcal{S}_{(L,(n-2)^\pm,b)} = \#\mathcal{S}_{(L,(n-2)^+,b)}$ if and only if $((-1)^\frac{n-2}{2}u_1 \cdots u_{n-2},\pi) = 1$. 
    Here we refer to Section \ref{sec:not} for the Hibert symbol $(-,-)$ and its property.
    Thus we have that
   \[ \chi((n-2)^\pm) \#\mathcal{S}_{(L,(n-2)^\pm,b)} = ((-1)^{\frac{n-2}{2}}{u}_1 \cdots {u}_{n-2},\pi) \cdot \#\mathcal{S}_{(L,(n-2)^\pm,b)}.
  \]
  This, together with Step 1, proves Proposition \ref{prop:Sab}.(2).

   If $n$ is odd, then we have to consider $a_0$ which is defined  in Lemma \ref{lem:rep_g_S_ab}.(3).
  This will be treated case by case in the proof.
For the further use, we state the following formula:
    \begin{equation}\label{eq:sign_n-1}
    \chi(n-1)^\pm)\#\mathcal{S}_{(L,(n-1)^{\pm},b)} 
    = 2 \#\mathcal{S}_{(L,(n-1)^{+},b)} - \#\mathcal{S}_{(L,(n-1)^{\pm},b)}.
    \end{equation}
    
\end{enumerate}    
\end{remark}

\textit{\textbf{Proof of Proposition \ref{prop:Sab}.}}  
We separate the cases according to the sign of $d_2-2(b-l)$, which appears in the term $\pi^{d_2-2(b-l)}v_2$ of Equations (\ref{eq:Sab}).(3) and  (\ref{eq:Sab2}).($3'$). 
This sign gives a bound for $l$.
Thus by applying  steps of Remark \ref{rmk:4.3} to each case, we obtain  partitions of $\# \mathcal{S}_{(L,a^\pm,b)}$,  $\# \mathcal{S}_{(L,a^+,b)}$,  and $\chi(a)\# \mathcal{S}_{(L,a^\pm,b)}$.
  An individual summand of the partition of $\#\mathcal{S}_{(L,a^\pm,b)}$ according to the sign
  is called a part of $\#\mathcal{S}_{(L,a^\pm,b)}$, and the same for $\# \mathcal{S}_{(L,a^+,b)}$ and for $\chi(a)\# \mathcal{S}_{(L,a^\pm,b)}$.

  \begin{enumerate} 
  \item The case that $d_2-2(b-l)>0$, equivalently $l>b-d_2/2$:
    
    By Equation (\ref{eq:Sab}).(1), we may and do assume that $b<(d_1+d_2)/2$.
    We put $w = 1$ in Equation (\ref{eq:form_x}).
    \begin{enumerate}[label=Step \arabic*:]
    \item 
Since $\pi^{d_2-2(b-l)}v_2 \in \pi\mathfrak{o}$,   Equations (\ref{eq:Sab}) and (\ref{eq:Sab2}) are equivalent to the followings:
  \begin{equation}\label{eq:Sab3}
    \begin{cases}
      (1): \max(0,b-(d_2-1)/2) \leq l \leq \min(d_1/2,b);  \\
      (2): s_i = 0 \textit{ for } i> \min(d_1/2,b)-l,
    \end{cases} \quad\begin{cases}
      (1'): l<d_1/2;  \\
      (2'): s_{d_1/2-l} = 0 .
    \end{cases}
  \end{equation}
Note that Equation (\ref{eq:Sab}) is equivalent to Equation (\ref{eq:Sab3}).(1)-(2).  
    By Step 1 of Remark \ref{rmk:4.3}, the part of $\# \mathcal{S}_{(L,(n-2)^\pm,b)}$ is the same as the number of $l$ and $s_i$'s satisfying Equation (\ref{eq:Sab3}), and the part of $\# \mathcal{S}_{(L,(n-1)^\pm,b)}$ is the same as the number of $l$ and $s_i$'s satisfying Equation (\ref{eq:Sab3}).(1)-(2) but not satisfying either Equation (\ref{eq:Sab3}).($1'$) or ($2'$).
    This discussion yields the followings:
      \begin{equation}\label{eq:form_Sab_b<d2/2}
    \begin{aligned}
    \textit{the part of }\# \mathcal{S}_{(L,(n-2)^\pm,b)} 
    &=
    \begin{cases}
        \sum_{i = 0}^b f^i
        &\textit{if } 0 \leq b <d_1/2; \\
        \sum_{i = 0}^{\lfloor (d_1-1)/2 \rfloor} f^i
        &\textit{if } d_1/2 \leq b < d_2/2; \\
        \sum_{i = 0}^{\lfloor (d_1-1)/2 \rfloor + \lfloor (d_2-1)/2 \rfloor} f^i
        &\textit{if } d_2/2 \leq b <(d_1+d_2)/2.
    \end{cases}\\
    \textit{the part of } \# \mathcal{S}_{(L,(n-1)^\pm,b)} &=
    \begin{cases}
        0
        &\textit{if } 0 \leq b <d_1/2; \\
        \delta(d_1) f^{d_1/2}
        &\textit{if } d_1/2 \leq b < d_2/2; \\
        \delta(d_1) f^{d_1/2 + \lfloor (d_2-1)/2 \rfloor - b}
        &\textit{if } d_2/2 \leq b <(d_1+d_2)/2.
    \end{cases}
    \end{aligned}
  \end{equation}

    \item Suppose that $n$ is odd and $a = n-1$.
    If $l$ satisfies Equation (\ref{eq:Sab3}).(1), then $l\leq d_1/2$. 
If $l = d_1/2$, then the (1,1)-entry of $A$ ($=a_0$) is a unit and its reduction modulo $\pi$ is $\overline{v}_1$.
  Otherwise (equivalently $l < d_1/2$), the (2,2)-entry of $A$ ($=a_0$) is a unit and 
 its reduction modulo $\pi$ is $\overline{v}_1\overline{s}_{d_1/2-l}^2$ with $\overline{s}_{d_1/2-l}\neq 0$.
  In both cases, $d_1$ must be even.
  By Lemma \ref{lem:rep_g_S_ab}.(3), one of the part of $\mathcal{S}_{(L,(n-1)^+,b)}$ or the part of  $\mathcal{S}_{(L,(n-1)^-,b)}$ is empty. 
  Then for $b<(d_1+d_2)/2$, the part of $\chi((n-1)^\pm) \#\mathcal{S}_{(L,(n-1)^\pm,b)}$ is 
 \begin{equation}\label{eq:sign_n-1_1}  
      \begin{aligned}
                 \delta(d_1)((-1)^{\frac{n-1}{2}}u_1 \cdots u_{n-2}v_1,\pi) \cdot \textit{the part of }\# \mathcal{S}_{(L,(n-1)^\pm,b)}, \\
       \textit{    where $n$ is odd and $\textit{the part of }\# \mathcal{S}_{(L,(n-1)^\pm,b)}$ follows Equation (\ref{eq:form_Sab_b<d2/2}).}
          \end{aligned}
\end{equation} 
    \end{enumerate}

    If $b<d_2/2$, then the given assumption in this case, which is $l>b-d_2/2$,  is unnecessary since $l$ should be non-negative. Thus  $\# \mathcal{S}_{(L,a^\pm,b)} = \textit{the part of }\# \mathcal{S}_{(L,a^\pm,b)}$ and $\chi((n-1)^\pm) \#\mathcal{S}_{(L,(n-1)^\pm,b)} = \textit{the part of }\chi((n-1)^\pm) \#\mathcal{S}_{(L,(n-1)^\pm,b)}$.
  This proves Proposition \ref{prop:Sab}.(1) and (3) for $b<d_2/2$. 
\\

   \item The case that $d_2-2(b-l)<0$, equivalently $l<b-d_2/2$:
   
Note that $\pi^{d_2-2(b-l)}v_2 \not\in \mathfrak{o}$.
    Since $l$ is non-negative by Equation (\ref{eq:Sab}).(1) and  Equation (\ref{eq:Sab}).(3) requires to have a solution, we may and do assume that $b >d_2/2$, that $d_1+d_2$ is even, and that $-\overline{v}_1 \overline{v}_2$ is a square in $\kappa^{\times}$. 
    Then we choose $w \in \mathfrak{o}^\times$ in Equation (\ref{eq:form_x}) such that $v_1w^2 + v_2 = 0$.
    
    On the other hand, we may and do assume that $b\leq (d_1+d_2)/2$ by Equation (\ref{equation:bd12}). 
    \begin{enumerate}[label=Step \arabic*:]
    \item

 Equations (\ref{eq:Sab}) and  (\ref{eq:Sab2}) are equivalent to the followings:
  \begin{equation}\label{eq:Sab4}
  \begin{cases}
      (1): 0 \leq l < b-d_2/2; \\
      (2): s_{(b-2l)-(d_2-d_1)/2} = \pm 1; \\
      (3) : s_i = 0 &\textit{for } i>(d_1+d_2)/2-b \textit{ and } i\neq(b-2l)-(d_2-d_1)/2; \\
      (4): s_{(d_1+d_2)/2-b} = 0 &\textit{if } b<(d_1+d_2)/2.
  \end{cases}
  \end{equation}
  Here  $(b-2l)-(d_2-d_1)/2>(d_1+d_2)/2-b \geq 0$.
  Note that Equation (\ref{eq:Sab}) is equivalent to Equation (\ref{eq:Sab4}).(1)-(3).
\begin{enumerate}
    \item 
    Suppose that  $b = (d_1+d_2)/2$.
    By Step 1 of Remark \ref{rmk:4.3}, the part of $\# \mathcal{S}_{(L,n^\pm,(d_1+d_2)/2)}$ is the same as the number of $l$ and $s_i$'s  satisfying Equation (\ref{eq:Sab4}).(1)-(3).
    This yields the following:
    \begin{equation}\label{eq:Sab_b>=d2/2_3'}
    \begin{aligned}
        &\textit{the part of } \# \mathcal{S}_{(L,n^\pm,(d_1+d_2)/2)}
        =\begin{cases}
            d_1 &\textit{if } d_1, d_2 \textit{ are even}; \\
            (d_1+1) &\textit{if } d_1, d_2 \textit{ are odd},
        \end{cases} \\
        &\textit{ where }  d_1 + d_2 \textit{ is even and } -v_1v_2 \textit{ is a square}.
    \end{aligned}
    \end{equation}
    
\item 
Suppose that  $d_2/2<b<(d_1+d_2)/2$.
By Step 1 of Remark \ref{rmk:4.3}, 
 the part of $\# \mathcal{S}_{(L,(n-2)^\pm,b)}$ is the same as the number of $l$ and $s_i$'s satisfying Equation (\ref{eq:Sab4}), and the part of $\# \mathcal{S}_{(L,(n-1)^\pm,b)}$ is the same as the number of $l$ and $s_i$'s satisfying Equation (\ref{eq:Sab4}).(1)-(3) but not satisfying Equation (\ref{eq:Sab4}).(4).
This discussion yields the followings:
  \begin{equation}\label{eq:Sab_b>=d2/2_3}
  \begin{aligned}
    \begin{cases}
      \textit{the part of } \# \mathcal{S}_{(L,(n-2)^\pm,b)} = (2b-2\lceil d_2/2 \rceil) f^{(d_1+d_2)/2-b-1}; \\
   \textit{the part of } \# \mathcal{S}_{(L,(n-1)^\pm,b)}  = (2b-2\lceil d_2/2 \rceil)(f-1) f^{(d_1+d_2)/2-b-1},
    \end{cases} \\
    \textit{where }  d_1+d_2 \textit{ is even, } -v_1v_2 \textit{ is a square, and } d_2/2 < b < (d_1+d_2)/2. 
    \end{aligned}
  \end{equation}
  \end{enumerate}

    \item
The case that $n$ is even and $a=n-2$ is treated in Step 2 of Remark \ref{rmk:4.3}.
Thus we work with the other case  when  $n$ is odd and $a = n-1$.
  Since the assumption that $l<b-d_2/2$ yields that $d_1>2l$, the (1,1)-entry of $A$ is not a unit.
  Then the (2,2)-entry of $A$ ($=a_0$) is a unit, and its reduction  modulo $\pi$ is $2 \overline{v}_1 \overline{s}_{(d_1+d_2)/2-b} \overline{w}^2 \neq 0$.
  Since $p$ is odd, exactly the half of the elements $\overline{s}_{(d_1+d_2)/2-b} \in \kappa^\times$ are squares.
  Thus Lemma \ref{lem:rep_g_S_ab}.(3) yields that
  \begin{equation}\label{eq:sign_n-1_3}
  \begin{aligned}
    &\textit{the part of } \chi((n-1)^\pm) \#\mathcal{S}_{(L,(n-1)^\pm,b)} = 0, \\
    &\textit{where $n$ is odd and }  d_2/2 < b < (d_1+d_2)/2.
    \end{aligned}
  \end{equation}
  \end{enumerate}
\textit{ }  
  
  \item The case that $d_2-2(b-l)=0$, equivalently $l = b-d_2/2$:
    
    This case occurs only when $d_2$ is even. For the further use of notation, we will multiply all formulas by $\delta(d_2)$.
    Here we refer to Section \ref{sec:not} for the notion of $\delta(-)$.
    
    Note that $\pi^{d_2-2(b-l)}v_2  \in \mathfrak{o}^\times$.
    By Equations (\ref{eq:Sab}).(1) and (\ref{equation:bd12}), we may and do assume that $d_2/2 \leq b \leq (d_1+d_2)/2$.
    We put $w = 1$ in Equation (\ref{eq:form_x}).

    \begin{enumerate}[label=Step \arabic*:]
    \item
    Equations (\ref{eq:Sab}) and   (\ref{eq:Sab2}) are equivalent to the followings:
    \begin{equation}\label{eq:Sab_b-d_2/2=l}
        \begin{cases}
            (1) : d_2/2 \leq b \leq (d_1+d_2)/2; \\
            (2) : s_i = 0 \textit{ for } i>(d_1+d_2)/2-b.
        \end{cases}, \quad
        \begin{cases}
            (1') : b \neq (d_1+d_2)/2; \\
            (2') : \overline{v}_1 \overline{s}_{(d_1+d_2)/2-b}^2 + \overline{v}_2 = 0.
        \end{cases}
    \end{equation}
Here $(2')$ is indeterminant if $d_1+d_2$ is odd or $-v_1v_2$ is not a square. This will be treated carefully below in Equation (\ref{eq:Sab_b-d2/2=l_total}).
Note that Equation (\ref{eq:Sab}) is equivalent to Equation (\ref{eq:Sab_b-d_2/2=l}).(1)-(2).

\begin{enumerate}
    \item 
    If $b = (d_1+d_2)/2$, then the part of $\# \mathcal{S}_{(L,n^\pm,(d_1+d_2)/2)}$ is the same as the number of $s_i$'s satisfying Equation (\ref{eq:Sab_b-d_2/2=l}).(1)-(2), which is 1.
    This case appears only when $d_1 + d_2$ is even, and thus we have that
    \begin{equation}\label{eq:Sab_b-d_2/2=l_2}
      \textit{the part of }\# \mathcal{S}_{(L,n^\pm,(d_1+d_2)/2)} = \delta(d_1+d_2)\delta(d_2) = \delta(d_1)\delta(d_2).
    \end{equation}

    \item If $d_2/2 \leq b<(d_1+d_2)/2$, then Equation (\ref{eq:Sab_b-d_2/2=l}).(1$'$) is redundant.
     By Step 1 of Remark \ref{rmk:4.3}, the part of $\# S_{(L,(n-2)^\pm,b)}$ is the same as the number of  $s_i$'s satisfying Equation (\ref{eq:Sab_b-d_2/2=l}), and the part of $\# S_{(L,(n-1)^\pm,b)}$ is the same as the number of  $s_i$'s satisfying Equation (\ref{eq:Sab_b-d_2/2=l}).(1)-(2) but not satisfying Equation (\ref{eq:Sab_b-d_2/2=l}).($2'$).
    Here Equation (\ref{eq:Sab_b-d_2/2=l}).(2$'$) has a solution if and only if $d_1+d_2$ is even and $-v_1v_2$ is a square.
This discussion yields the followings:
    {\small
    \begin{equation}\label{eq:Sab_b-d2/2=l_total}
    \begin{aligned}
      &\textit{the part of }\# \mathcal{S}_{(L,(n-2)^\pm,b)}=
      \begin{cases}
          0
          &\textit{if } d_1+d_2 \textit{ is odd or } -v_1 v_2 \notin (\mathfrak{o}^\times)^2; \\
          \delta(d_2) 2 f^{(d_1+d_2)/2-b-1}
          &\textit{if } d_1+d_2 \textit{ is even and } -v_1 v_2 \in (\mathfrak{o}^\times)^2,
      \end{cases}
       \\
      &\textit{the part of } \# \mathcal{S}_{(L,(n-1)^\pm,b)} =
      \begin{cases}
          \delta(d_2)f^{\lfloor (d_1+d_2)/2 \rfloor-b}
          &\textit{if } d_1+d_2 \textit{ is odd or } -v_1 v_2 \notin (\mathfrak{o}^\times)^2; \\
          \delta(d_2) (f-2) f^{(d_1+d_2)/2-b-1}
          &\textit{if } d_1+d_2 \textit{ is even and } - v_1 v_2 \in (\mathfrak{o}^\times)^2,
      \end{cases} \\
      &\textit{where } d_2/2 \leq b <(d_1+d_2)/2.
    \end{aligned}
    \end{equation}
    }
\end{enumerate}    

    \item Suppose that $n$ is odd and $a = n-1$.
    Since  $b<(d_1+d_2)/2$ by Equation (\ref{equation:bd12}),
 the (1,1)-entry of $A$ is not a unit.
    Then the (2,2)-entry of $A$ ($= a_0$) is a unit, and its reduction modulo $\pi$ is   $\overline{v}_1 \overline{s}_{(d_1+d_2)/2-b}^2 + \overline{v}_2 \neq 0$.
    Here we understand $s_{(d_1+d_2)/2-b}=0$ if $d_1+d_2$ is odd. 
     By Lemma \ref{lem:rep_g_S_ab}.(3), the part of $\# \mathcal{S}_{(L,(n-1)^+,b)}$ is the number of $s_i$'s satisfying Equation (\ref{eq:Sab_b-d_2/2=l}).(1)-(2)  such that
     \[ \tag{$\ast$} (-1)^{\frac{n-1}{2}} \overline{u}_1 \cdots \overline{u}_{n-2} (\overline{v}_1 \overline{s}_{(d_1+d_2)/2-b}^2 + \overline{v}_2)
  \]
     is contained in $(\kappa^\times)^2$.
 Since the above $(\ast)$ only concerns $\overline{s}_{(d_1+d_2)/2-b}$, not involving  $s_i$'s with $1\leq i <(d_1+d_2)/2-b$, 
 we have that 
 \begin{equation}\label{eq:Sab_b-d_2/2=l_(n-1)+}
    \begin{aligned}
     \textit{the part of } \# \mathcal{S}_{(L,(n-1)^+,b)} 
     =\delta(d_2) \cdot \# \{\overline{s}_{(d_1+d_2)/2-b} \in \kappa \mid (\ast) \in (\kappa^\times)^2\}  \cdot f^{\lfloor (d_1+d_2-1)/2 \rfloor -b}.
     \end{aligned}
    \end{equation}

To simplify notation,   we let 
    \[
    \mathcal{U}:=(-1)^{\frac{n-1}{2}} \overline{u}_1 \cdots \overline{u}_{n-2}\overline{v}_1, \mathcal{V}:=(-1)^{\frac{n-1}{2}} \overline{u}_1 \cdots \overline{u}_{n-2} \overline{v}_2 \in \kappa^{\times}.     
    \]

If $d_1+d_2$ is odd so that $s_{(d_1+d_2)/2-b}=0$, then we have that 
  \begin{equation}\label{eq:Sab_b-d2/2=l_notsquare_+1}
      \textit{the part of }\# \mathcal{S}_{(L,(n-1)^+,b)}=
      \begin{cases}
        0
        &\textit{if } \mathcal{V} \notin (\kappa^\times)^2; \\
        \delta(d_2)\cdot f^{(d_1+d_2-1)/2-b}
        &\textit{if } \mathcal{V} \in (\kappa^\times)^2.
         \end{cases}
    \end{equation}

If $d_1+d_2$ is even,
     then Lemma \ref{lem:norm}, which will be stated below, yields the following:
  \begin{equation}\label{eq:Sab_b-d2/2=l_notsquare_+}
      \textit{the part of }\# \mathcal{S}_{(L,(n-1)^+,b)}=
      \begin{cases}
        \delta(d_2)\frac{(f-1)}{2}\cdot f^{(d_1+d_2)/2-b-1}
        &\textit{if } -\mathcal{V} \notin (\kappa^\times)^2\\
        \delta(d_2)\frac{(f-3)}{2}\cdot f^{(d_1+d_2)/2-b-1}
        &\textit{if } \mathcal{U}, -\mathcal{U}\mathcal{V} \in (\kappa^\times)^2 \\
        \delta(d_2)\frac{(f+1)}{2}\cdot f^{(d_1+d_2)/2-b-1}
        &\textit{if } \mathcal{U}, -\mathcal{U}\mathcal{V} \notin (\kappa^\times)^2.
      \end{cases}
    \end{equation}
Here  $-\mathcal{V} \notin (\kappa^\times)^2$ if and only if exactly one of $\mathcal{U}, -\mathcal{U}\mathcal{V}$ is in $(\kappa^\times)^2$.

By plugging in Equations (\ref{eq:Sab_b-d2/2=l_total}), (\ref{eq:Sab_b-d2/2=l_notsquare_+1}), and (\ref{eq:Sab_b-d2/2=l_notsquare_+})   into  Equation (\ref{eq:sign_n-1}), we have  that 
z
    \begin{equation}\label{eq:Sab_b-d2/2=l_total_pm}
    \begin{aligned}
      &\textit{the part of }  \chi((n-1)^\pm) \#\mathcal{S}_{(L,(n-1)^\pm,b)} \\
     &=
     \begin{cases}
        \delta(d_2) ((-1)^{\frac{n-1}{2}}u_1 \cdots u_{n-2}v_2,\pi) f^{(d_1+d_2-1)/2-b}
        &\textit{if } d_1+d_2 \textit{ is odd}; \\
        -\delta(d_2)((-1)^{\frac{n-1}{2}}u_1 \cdots u_{n-2}v_1,\pi) f^{(d_1+d_2)/2-b-1}
        &\textit{if } d_1+d_2 \textit{ is even},
     \end{cases} \\
     &\textit{where $n$ is odd and } d_2 \leq b <(d_1+d_2)/2.
     \end{aligned}
    \end{equation}

    Here we use the fact  that $-v_1v_2 \in (\mathfrak{o}^\times)^2$  if and only if $-\mathcal{U}\mathcal{V} \in (\kappa^\times)^2$.
    \end{enumerate}
\end{enumerate}
\textit{ }

  We finally combine all of the above formulas to calculate $\# \mathcal{S}_{(L,a^\pm,b)}$, $\chi((n-2)^\pm)\# \mathcal{S}_{(L,(n-2)^\pm,b)}$, and $\chi((n-1)^\pm)\# \mathcal{S}_{(L,(n-1)^\pm,b)}$.
  
  \begin{itemize}
    \item If $0 \leq b <d_2/2$, then 
    $\#\mathcal{S}_{(L,a^\pm,b)}$ and  $\chi((n-1)^\pm)\# \mathcal{S}_{(L,(n-1)^\pm,b)}$ are given by the formulas in Equations (\ref{eq:form_Sab_b<d2/2}) and (\ref{eq:sign_n-1_1}) respectively.
    This follows from the discussion in the paragraph following Equation (\ref{eq:sign_n-1_1}), and
    thus yields the desired formulas for the first two cases of $\# \mathcal{S}_{(L,(n-2)^\pm,b)}$, $\# \mathcal{S}_{(L,(n-1)^\pm,b)}$, and $\chi((n-1)^\pm)\# \mathcal{S}_{(L,(n-1)^\pm,b)}$ in Proposition \ref{prop:Sab}.(1) and (3).

    \item If $d_2/2 \leq b < (d_1+d_2)/2$, then $\#\mathcal{S}_{(L,a^\pm,b)}$ and $\chi((n-1)^\pm)\# \mathcal{S}_{(L,(n-1)^\pm,b)}$ are given by summing up their parts associated to $d_2-2(b-l)>0$, $d_2-2(b-l)<0$, and $d_2-2(b-l) = 0$ respectively.
    The parts of $\#\mathcal{S}_{(L,a^\pm,b)}$ are described in Equations (\ref{eq:form_Sab_b<d2/2}), (\ref{eq:Sab_b>=d2/2_3}), and (\ref{eq:Sab_b-d2/2=l_total}), and the parts of $\chi((n-1)^\pm)\# \mathcal{S}_{(L,(n-1)^\pm,b)}$ are described in Equations (\ref{eq:sign_n-1_1}), (\ref{eq:sign_n-1_3}), and (\ref{eq:Sab_b-d2/2=l_total_pm}).
Here   the case $d_2-2(b-l) < 0$ does not appear when $d_2/2=b$. 
    This yields the desired formulas for the third and fourth cases of $\# \mathcal{S}_{(L,(n-2)^\pm,b)} $ and $\# \mathcal{S}_{(L,(n-1)^\pm,b)} $, and the third, fourth, and fifth cases of $\chi((n-1)^\pm)\# \mathcal{S}_{(L,(n-1)^\pm,b)}$ in Proposition \ref{prop:Sab}.(1) and (3).

    \item If $(d_1+d_2)/2 \leq b$, then $\#\mathcal{S}_{(L,a^\pm, b)}=0$ for $a=n-2$ or $a=n-1$ and $\chi((n-1)^\pm)\# \mathcal{S}_{(L,(n-1)^\pm,b)} = 0$ by  Equation (\ref{equation:bd12}).
 This yields the desired formulas for the last case of $\# \mathcal{S}_{(L,(n-2)^\pm,b)}$, the first case of $\# \mathcal{S}_{(L,(n-1)^\pm,b)}$, and 
the first case of  $\chi((n-1)^\pm)\# \mathcal{S}_{(L,(n-1)^\pm,b)}$  in Proposition \ref{prop:Sab}.(1) and (3).

    \item Equation (\ref{equation:bd12})  yields the desired formula for the first case of $\# \mathcal{S}_{(L,n^\pm,b)}$ in Proposition \ref{prop:Sab}.(1).

\item  $\#\mathcal{S}_{(L,n^\pm,(d_1+d_2)/2)}$ is the sum of its parts associated to  $d_2-2(b-l) < 0$ and $d_2-2(b-l)=0$, which are described in Equation (\ref{eq:Sab_b>=d2/2_3'}) and (\ref{eq:Sab_b-d_2/2=l_2}) respectively.
Here   the case $d_2-2(b-l) > 0$ does not appear.
    This yields the desired formula for  the second and third cases of  $\#\mathcal{S}_{(L,n^\pm,(d_1+d_2)/2)}$ in Proposition \ref{prop:Sab}.(1).

    \item Proposition \ref{prop:Sab}.(2), the formula of $\chi((n-2)^\pm)\# \mathcal{S}_{(L,(n-2)^\pm,b)}$,   is proved in Step 2 of Remark \ref{rmk:4.3}.   
      \qed
  \end{itemize}
\begin{Lem}\label{lem:norm}
Let $\mathcal{U}$ and  $\mathcal{V}$ be units in $\kappa$. 
Let $P_s(t):=t^2-\mathcal{U}s^2$ in $\kappa[t]$, where $s\in \kappa$. 
Then 
\[
\#\{s \in \kappa \mid P_s(t)=\mathcal{V}  \textit{ has a solution in $\kappa^{\times}$}\}
=\begin{cases}
    (f-1)/2 &\textit{if } -\mathcal{V} \notin (\kappa^\times)^2; \\
    (f-3)/2 &\textit{if } \mathcal{U}, -\mathcal{U}\mathcal{V} \in (\kappa^\times)^2; \\
    (f+1)/2 &\textit{if } \mathcal{U}, -\mathcal{U}\mathcal{V} \notin (\kappa^\times)^2.
\end{cases}
\]
\end{Lem}
\begin{proof}
    If $(t,s) \in \kappa^\times \times \kappa$ satisfies that $P_s(t) = \mathcal{V}$, then so does $(-t,s)$.
    Here $t \neq -t$ since $p$ is odd and $t \in \kappa^\times$.
    This yields that
 \[   \#\{s \in \kappa \mid P_s(t)=\mathcal{V}  \textit{ has a solution in $\kappa^{\times}$}\}
    = \frac{1}{2}\#\{(t,s) \in \kappa^\times \times \kappa \mid P_s(t) = \mathcal{V}\}.
   \]
    Thus it suffices to compute the right hand side of the above equation.
   Let $N : \kappa(\sqrt{\mathcal{U}}) \to \kappa$ be the norm map.
   Since $N(t + s\sqrt{\mathcal{U}}) = P_s(t)$ and $N$ is surjective, we have that
    \begin{equation}\label{eq:4.4.2}
    \#\{(t,s) \in \kappa \times \kappa \mid P_s(t) = \mathcal{V}\} = \# N^{-1}(\mathcal{V}) = 
    \begin{cases}
        f-1 &\textit{if }                                      \mathcal{U} \in (\kappa^\times)^2; \\
        f+1 &\textit{if } \mathcal{U} \notin (\kappa^\times)^2.
      \end{cases}
    \end{equation}
    Here  we use the fact that $\kappa(\sqrt{\mathcal{U}})$ is a field extension of $\kappa$ if and only if $\mathcal{U} \not\in (\kappa^\times)^2$.
Since $$\{(t,s) \in \kappa \times \kappa \mid P_s(t) = \mathcal{V}\}=\{(t,s) \in \kappa^\times \times \kappa \mid P_s(t) = \mathcal{V}\}\bigsqcup \{s \in \kappa \mid P_s(0) = \mathcal{V}\},$$
it suffices to compute the cardinality of the last set.    
    Here $P_s(0) = \mathcal{V}$ if and only if $s^2 = -\mathcal{U}^{-1}\mathcal{V}$.
    Then we easily have the following formulas:
    \[
\#\{s \in \kappa \mid P_s(0) = \mathcal{V}\}=        \begin{cases}
        2 &\textit{if }       -\mathcal{U}\mathcal{V}\in (\kappa^\times)^2; \\
        0 &\textit{if } -\mathcal{U}\mathcal{V}\not\in (\kappa^\times)^2.
      \end{cases}
    \]
    The subtraction of the above from Equation (\ref{eq:4.4.2}), divided by $2$, then yields the desired formula. 
\end{proof}

\subsection{Explicit values of  coefficients of the Siegel series}

In order to use Corollaries  \ref{cor:coeff_n0=n-1}-\ref{cor:coeff_n0=n-2} based on Proposition \ref{prop:Sab}, we  need to reinterpret $\mathfrak{e}_B$ and $\xi_B$ in terms of $u_i, v_i$, and $d_i$.
By Equations (\ref{eq:gke}) and (\ref{eq:gk12}), we have that
\[
\mathfrak{e}_B
= \begin{cases}
    2 \lfloor (d_1+d_2)/2 \rfloor &\textit{if } n \textit{ is even}; \\
    d_1+d_2 &\textit{if } n \textit{ is odd}. 
\end{cases}
\]
When $n$ is even, \cite[Proposition 5.5]{Cas11} yields that
  \begin{equation}\label{eq:xi_B}
  \xi_B =
  \begin{cases}
    ((-1)^{\frac{n}{2}} {u}_1 \cdots{u}_{n-2}v_1v_2 ,\pi) &\textit{if } d_1+d_2 \textit{ is even}; \\
    0 &\textit{if } d_1+d_2 \textit{ is odd}.
  \end{cases}
  \end{equation}
  Note that $n_0 = n-1$ if and only if $d_1 = 0$.
  
We finally obtain the following theorem by substituting formulas in Proposition \ref{prop:Sab} to Corollaries  \ref{cor:coeff_n0=n-1}-\ref{cor:coeff_n0=n-2}:

\begin{Thm}\label{thm45}
  The coefficients $c_t$ of $\tilde{F}_L(X)$ are as follows:
  \begin{enumerate}
      \item Suppose that $n_0 = n-1$.

      \begin{enumerate}
      \item If $n$ is even, then we have that
      \[
      \begin{cases}
        c_{2s+1} =
        -f^{-1/2}\xi_B
        &\textit{if } 0 \leq s<\frac{\mathfrak{e}_B-1}{2}; \\ 
        c_{2s} = 1
        &\textit{if } 0 \leq s<\frac{\mathfrak{e}_B}{2}; \\
        c_{\mathfrak{e}_B} = 1.
      \end{cases}
      \]
      \item If $n$ is odd, then we have that
      \[
      \begin{cases}
        c_{2s+1} = ((-1)^{\frac{n-1}{2}} {u}_1 \cdots{u}_{n-2} {v}_1,\pi)
        &\textit{if } 0 \leq s<\frac{\mathfrak{e}_B-1}{2}; \\
        c_{2s} = 1
        &\textit{if } 0 \leq s<\frac{\mathfrak{e}_B}{2}; \\
        c_{\mathfrak{e}_B} = \begin{cases}
            1 &\textit{if } \mathfrak{e}_B \textit{ is even}; \\
            ((-1)^{\frac{n-1}{2}} {u}_1 \cdots{u}_{n-2} {v}_1,\pi)  &\textit{if } \mathfrak{e}_B \textit{ is odd}.
        \end{cases}
      \end{cases}
      \]
      \end{enumerate}

        \item Suppose that $n_0 = n-2$.
        \begin{enumerate}
        
            \item If $n$ is even, then we have that
        \begin{equation*}
        \begin{aligned}
        &c_{2s+1} =
        \begin{cases}
        \left(((-1)^{\frac{n-2}{2}} {u}_1 \cdots{u}_{n-2} ,\pi) f^{\frac{1}{2}} -\xi_Bf^{-\frac{1}{2}} \right)\left(\sum\limits_{i = 0}^s f^{i}\right)
        &\textit{if }   0 \leq s<\frac{d_1}{2}; \\
        ((-1)^{\frac{n-2}{2}} {u}_1 \cdots{u}_{n-2} ,\pi))\cdot \sum\limits_{i = 0}^{\lfloor \frac{d_1-1}{2} \rfloor} f^{i+\frac{1}{2}} -  \xi_B \sum\limits_{i = 0}^{\lfloor \frac{d_1}{2} \rfloor} f^{i-\frac{1}{2}}
        &\textit{if } \frac{d_1}{2} \leq s < \frac{d_2}{2}; \\
        \left(((-1)^{\frac{n-2}{2}} {u}_1 \cdots{u}_{n-2} ,\pi) f^{\frac{1}{2}} -\xi_Bf^{-\frac{1}{2}} \right)\left(\sum\limits_{i = 0}^{\frac{\mathfrak{e}_B}{2}-s-1}f^i\right)  &\textit{if } \frac{d_2}{2} \leq s < \frac{\mathfrak{e}_B-1}{2} ,
        \end{cases} \\
        &c_{2s} =
        \begin{cases}
        \sum\limits_{i = 0}^s f^{i} - \xi_B((-1)^{\frac{n-2}{2}} {u}_1 \cdots{u}_{n-2} ,\pi))\left(\sum\limits_{i = 0}^{s-1} f^{i}\right) &\textit{if }   0 \leq s< \frac{d_1}{2}; \\
        \sum\limits_{i = 0}^{\lfloor \frac{d_1}{2} \rfloor} f^i - \xi_B((-1)^{\frac{n-2}{2}} {u}_1 \cdots{u}_{n-2} ,\pi))  \sum\limits_{i = 0}^{\lfloor \frac{d_1-1}{2} \rfloor} f^i  &\textit{if } \frac{d_1}{2} \leq s < \frac{d_2}{2}; \\
            \sum\limits_{i = 0}^{\frac{\mathfrak{e}_B}{2}-s} f^i -\xi_B((-1)^{\frac{n-2}{2}} {u}_1 \cdots{u}_{n-2} ,\pi)) \sum\limits_{i = 0}^{\frac{\mathfrak{e}_B}{2}-s-1} f^i
            &\textit{if } \frac{d_2}{2} \leq s <\frac{\mathfrak{e}_B}{2},
        \end{cases}\\
        &c_{\mathfrak{e}_B} = 1.
        \end{aligned}
        \end{equation*}

      \item If $n$ is odd, then we have that
      \begin{equation}
            c_{2s+1}
            = \begin{cases}
              0 &\textit{if } 0 \leq s <d_1/2; \\
              ((-1)^{\frac{n-1}{2}} {u}_1 \cdots{u}_{n-2} {v}_1,\pi)\delta(d_1)f^{\frac{d_1}{2}} &\textit{if } \frac{d_1}{2} \leq s < \frac{d_2}{2}; \\
              0
              &\textit{if } \frac{d_2}{2} \leq s < \frac{\mathfrak{e}_B-1}{2},
              \,d_1, d_2 \textit{ are even}; \\
              0
              &\textit{if } \frac{d_2}{2} \leq s < \frac{\mathfrak{e}_B-1}{2},
              \,d_1, d_2 \textit{ are odd}; \\
              ((-1)^{\frac{n-1}{2}} {u}_1 \cdots{u}_{n-2} {v}_1,\pi)f^{\frac{\mathfrak{e}_B-1}{2} -s}  
              &\textit{if } \frac{d_2}{2} \leq s < \frac{\mathfrak{e}_B-1}{2},
              d_1 \textit{ is even, } d_2 \textit{ is odd}; \\
              ((-1)^{\frac{n-1}{2}} {u}_1 \cdots{u}_{n-2} {v}_2,\pi)f^{\frac{\mathfrak{e}_B-1}{2} -s}       
              &\textit{if } \frac{d_2}{2} \leq s < \frac{\mathfrak{e}_B-1}{2}, 
              \,d_1 \textit{ is odd, } d_2 \textit{ is even}, \\
            \end{cases}
    \end{equation}
    \begin{equation}
            c_{2s}
            =
            \begin{cases}
               f^s &\textit{if } 0 \leq s <\frac{d_1}{2}; \\
              \delta(d_1)f^{d_1/2} &\textit{if } \frac{d_1}{2} \leq s < \frac{d_2}{2}; \\
              f^{\frac{\mathfrak{e}_B}{2}-s}
              &\textit{if } \frac{d_2}{2} \leq s < \frac{\mathfrak{e_B}}{2}, \, d_1, d_2 \textit{ are even}; \\
              (-v_1v_2,\pi)f^{\frac{\mathfrak{e}_B}{2}-s}
              &\textit{if } \frac{d_2}{2} \leq s < \frac{\mathfrak{e}_B}{2}, \,
              d_1, d_2 \textit{ are odd}; \\ 
              0
              &\textit{if } \frac{d_2}{2} \leq s < \frac{\mathfrak{e}_b}{2},
              d_1 \textit{ is even}, d_2 \textit{ is odd}; \\
              0
              &\textit{if } \frac{d_2}{2} \leq s < \frac{\mathfrak{e}_B}{2},
              d_1 \textit{ is odd}, d_2 \textit{ is even},
             \end{cases}          
        \end{equation}
        \[
        c_{\mathfrak{e}_B} = 
        \begin{cases}
          1
          &\textit{if } d_1, d_2 \textit{ are even}; \\
          (-v_1v_2,\pi)
          &\textit{if } d_1, d_2 \textit{ are odd}; \\
              ((-1)^{\frac{n-1}{2}} {u}_1 \cdots{u}_{n-2} {v}_1,\pi)
              &\textit{if }d_1 \textit{ is even, } d_2 \textit{ is odd}; \\
              ((-1)^{\frac{n-1}{2}} {u}_1 \cdots{u}_{n-2} {v}_2,\pi)
              &\textit{if }d_1 \textit{ is odd, } d_2 \textit{ is even}. \\
        \end{cases}
        \]
  \end{enumerate}
  \end{enumerate}
\end{Thm}

In order to prove the functional equation which is stated in Theorem \ref{thm:fe_coeff}, we need to reinterpret $\zeta_B$.
Recall that $\zeta_B = \eta_B$ or $1$ according as $n$ is odd or even respectively (cf. Theorem \ref{fe}).
    When $n$ is odd, \cite[Definition 0.4]{IK1} and \cite[Example 63:12]{MT00} yield that
  \begin{equation*}\label{eq:eta_B}
  \eta_B = 
  \begin{cases}
  1 &\textit{if } d_1, d_2 \textit{ are even}; \\
  (-{v}_1 {v}_2,\pi) &\textit{if }d_1, d_2 \textit{ are odd}; \\
  ((-1)^{\frac{n-1}{2}} {u}_1 \cdots{u}_{n-2}{v}_1,\pi) &\textit{if } d_1  \textit{ is even, } d_2 \textit{ is odd}; \\
  ((-1)^{\frac{n-1}{2}} {u}_1 \cdots {u}_{n-2} {v}_2,\pi) &\textit{if } d_1  \textit{ is odd, } d_2 \textit{ is even}.
  \end{cases}
  \end{equation*}
  This, together with  Theorem \ref{thm45}, yields the functional equation of the Siegel series stated in  Theorem \ref{thm:fe_coeff} as follows:
      \begin{Cor}\label{cor:functionaleq}
  We have that
  \[
  c_t = \zeta_B \cdot  c_{\mathfrak{e}_B-t} \textit{ for } 0 \leq t \leq \mathfrak{e}_B.
  \]
\end{Cor}

\bibliographystyle{alpha}
\bibliography{References}

\end{document}